\documentclass[12pt,a4paper]{scrartcl}

\usepackage[utf8]{inputenc}
\usepackage[english]{babel}
\usepackage[backend=bibtex, style=alphabetic, maxnames=10]{biblatex} 
\usepackage{fancyhdr}
\usepackage{ragged2e}
\usepackage{amsthm}
\usepackage{amsmath}
\usepackage{amsfonts}
\usepackage{amssymb}
\usepackage{mathtools}
\usepackage[colorinlistoftodos]{todonotes}

\usepackage{tikz-cd}

\author{Pernak, Rauzy}

\newcommand{\N}{\mathbb{N}}
\newcommand{\Z}{\mathbb{Z}}
\newcommand{\Lan}{\mathcal{L}}

\DeclarePairedDelimiterX\set[1]\lbrace\rbrace{#1} 

\newtheoremstyle{}{}{}{}{}{}{}{}{\thmname{#1}\thmnumber{ #2} #3} 
\newcounter{stmt} 

\newtheorem{lemma}[stmt]{Lemma}
\theoremstyle{definition}
\newtheorem{definition}[stmt]{Definition}
\newtheorem{theorem}[stmt]{Theorem}
\newtheorem{theoremalpha}{Theorem}
\newtheorem{propositionalpha}[theoremalpha]{Proposition}

\newtheorem{example}[stmt]{Example}
\newtheorem{remark}[stmt]{Remark}
\newtheorem{proposition}[stmt]{Proposition}

\newtheorem{corollary}[stmt]{Corollary}
\newtheorem{problem}[stmt]{Problem}

\begin{document}
	\pagenumbering{gobble}	
	\title{Groups with presentations in EDT0L}
\author{Laurent Bartholdi, Leon Pernak, Emmanuel Rauzy}
\maketitle

\begin{abstract}
	To any family of languages LAN, let us associate the class, denoted $\pi(\text{LAN})$, of finitely generated groups that admit a group presentation whose set of relators forms a language in LAN. We show that the class of L-presented groups, as introduced by the first author in 2003, is exactly the class of groups that admit presentations in the family of languages EDT0L. 
	We show that the marked isomorphism problem is not semi-decidable for groups given by EDT0L presentations,  contrary to the finite presentation case. 
	We then extend and unify results of the first author with Eick and Hartung about nilpotent and finite quotients, by showing that it is possible to compute the marked hyperbolic and marked abelian-by-nilpotent quotients of a group given by an EDT0L presentation. 
	Finally, we show how the results about quotient computations enable  the construction of  recursively presented groups that do not have EDT0L presentations, thus proving $\pi(\text{EDT0L})\ne \pi(\text{REC})$. This is done by building a residually nilpotent group with solvable word problem whose sequence of maximal nilpotent quotients is non-computable.
\end{abstract}

\lhead{EDT0L Presentations}
\rhead{Bartholdi, Pernak, Rauzy}

\begin{justify}
	\pagenumbering{arabic}
	\pagestyle{fancy}
	
\section{Introduction}

	In the study of finitely generated groups, it is very natural to represent group elements as words over the generators. Many connections between group theory and the theory of formal languages stem from this simple observation. \\
	Such a connection, that has perhaps been little studied, appears via the notion of presentation of groups. For LAN a family of languages, we define a class of groups, which we denote by $\pi(\text{LAN})$, which corresponds to \emph{the class of groups that have some presentation in LAN}. Namely, it is the class of groups that can be written, for a certain alphabet $S$, as $$\mathbb{F}_S/\langle\langle R\rangle \rangle, $$  where $\mathbb{F}_S$ denotes the free group over $S$, $R$ is a language in $\text{LAN}$ over the alphabet $S \cup S^{-1}$, and $\langle\langle R\rangle \rangle$ is the normal closure of $R$ in $\mathbb F_S$.\\
	In fact, presentations of groups are naturally studied in the category of marked groups: a  \emph{marked group} $(G,S)$ is a pair formed by a finitely generated group together with a generating tuple of this group. When a group $G$ is given as a quotient of a free group  $\mathbb{F}_S$, we naturally get a marking of $G$ by taking as a generating tuple the image of the basis $S$ of the free group in $G$. 
	This leads us to define, associated to a family of languages LAN, a class of marked groups  denoted $\pi_M(\text{LAN})$, which is the class of marked groups that admit presentations in LAN.  
	\\
	Two classes of groups defined via presentations in families of languages are well studied: \emph{finitely presentable groups} and \emph{recursively presentable groups}. These correspond  respectively to $\pi(\text{FIN})$  and $\pi(\text{RE})$, where  FIN is the family of finite languages and RE the family of recursively enumerable languages. 
	They are related via Higman's embedding theorem \cite{Higman1961}, which says that the set of finitely generated subgroups of finitely presentable groups is exactly the set of recursively presentable groups. 
	\\
	However, most of the theory of formal languages takes place \emph{between} the families FIN and RE: FIN and RE could be considered respectively the smallest and biggest families of languages usually encountered in formal language theory. 
	\\
	Because of this, many presentations that are commonly given as typical examples of recursive presentations in fact provide much more information than an ``average'' recursive presentation. 
	For example, the lamplighter group $\mathbb{Z}/2\wr\mathbb{Z}$ is given by the presentation 
	$$\langle a,\varepsilon\vert\,\varepsilon^{2},\,\left[\varepsilon,a^{i}\varepsilon a^{-i}\right],i\in\mathbb{Z}\rangle$$
	From this given presentation, one immediately sees that the abelianization of the lamplighter group is $\mathbb{Z}/2\times\mathbb{Z}$. But compare this to the following result of the third author: 
	\begin{theorem}[\cite{Rauzy2021}]
		If P is a group property which is semi-decidable for groups given by recursive presentations, then P is \emph{quotient stable} (that is to say, preserved under taking quotients).
	\end{theorem}
	Since ``having $H$ as a quotient'' is quotient stable exactly when $H$ is the trivial group, \emph{recursive presentations do not permit to compute quotients,} and in particular the abelianization. In fact, while the set of relators of the presentation of the lamplighter group given above is indeed recursively enumerable, it also belongs to many usually encountered families of languages that are smaller than RE (including EDT0L). \\ 
	These remarks point to a lack of vocabulary about group presentations; this article was written to address this fact. \\
	
	\vspace{0.3cm}
	
	In \cite{Lysenok1985}, Lysenok proved that the Grigorchuk group, which was proven by Grigorchuk himself not to have a finite presentation, still admits a simple enough presentation, obtained by starting from three relators and iterating a simple letter substitution to these relators to obtain an infinite presentation:
$$\langle a,c,d\vert\, \sigma ^i (a^2), \sigma ^i ((ad)^4), \sigma ^i ((adacac)^4), i\ge 0 \rangle,$$ 
where $\sigma \{a,c,d\}^*\rightarrow\{a,c,d\}^*$ is defined by $\sigma (a)=aca$, $\sigma(c)=cd$ and $\sigma(d)=c$. 
	This result led the first author to introduce the notion of \emph{L-presentation} in \cite{Bartholdi2003}. A group $G$ with generating set $S$ has a finite L-presentation if there are two finite sets of relations  $Q,R \subseteq \mathbb{F}_S$ and a finite set  $\Phi$ of endomorphisms of $\mathbb{F}_S$ such that $G$ is isomorphic to the quotient
	$$\mathbb{F}_S/ \langle \langle Q \cup \bigcup_{\phi\in \Phi ^* }\phi (R)\rangle \rangle,$$ where $\Phi ^{*}$ designates the monoid generated by $\Phi$, i.e. the closure of $\{1\}\cup \Phi$ under composition. 
	If $Q$ can be supposed empty in the above, $G$ is said to have an \emph{ascending L-presentation}. \\
	This notion yields a wide generalization of Lysenok's result: 
	\begin{theorem}[\cite{Bartholdi2003}]
		Any finitely generated self-similar regular contracting branch group admits a finite L-presentation, but no finite presentation. 
	\end{theorem}
	Several other examples of finitely L-presented groups were given in \cite{Bartholdi2003}, including wreath products and free groups in varieties given by finitely many identities. \\
	The notion of L-presentation has since found use in several distinct applications, for instance because they often yield finitely presented groups via HNN extentions, see \cite{Sapir2002, Olshanskii2002}, or to let us investigate branch groups \cite{Hartung2013, DiDomenico2022}; for more applications see \cite{Funar2010}. 
	
	\subsection{First result: identification of L-presented groups and $\pi(\text{EDT0L})$}
	
	In view of our previous remarks about formal languages, it seems desirable  to relate the notion of L-presentation to a notion of presentation in a certain family of languages. 
	It is rather easy to see that  L-presented groups coincide with groups that admit a presentation in DTF0L+FIN (see Section \ref{Subsec: Def L languages} for a definition), but this family of languages is badly behaved in terms of closure properties, and not considered to be a very interesting family of languages \cite{Rozenberg1997}. \\
	But different families of languages can yield the same classes of groups: a group has a context-free presentation if and only if it has a finite presentation (Proposition \ref{prop: Context free= finite}), and similarly a group has a recursive set of relators if and only if it has a  recursively enumerable set of relators (Craig's trick). Thus we can search for a family of languages that defines the same set of groups as DTF0L+FIN but which is better behaved. 
	\\
	Our first result is:
	\begin{theoremalpha}\label{Thm: EDT0L = Lpres intro}
		A group admits a presentation in EDT0L if and only if it has a finite L-presentation. Or again: $$\pi (\text{EDT0L})=\pi (\text{DTF0L+FIN}).$$
	\end{theoremalpha}
		The family of languages EDT0L is very well behaved in several respects. 
		Recent results of Ciobanu, Dieckert and Elder have opened the way for a wide range of applications of the EDT0L family of languages in group theory, starting with the study of solutions sets to equations in free and hyperbolic groups: it is shown in 	\cite{Ciobanu2021} that solution sets to systems of equations and inequations in hyperbolic groups form an EDT0L language whose grammar can be computed in  NPSPACE (extending \cite{Ciobanu2015} where similar results were proven for free groups).  
		
	Theorem \ref{Thm: EDT0L = Lpres intro} is thus the one that really anchors the study of finitely L-presented groups in the theory of formal languages. Moving forwards, it seems that the definition of this class of groups in terms of EDT0L languages should be the preferred one. \\
	Thanks to this characterization, results from the theory of formal languages can be applied to obtain results about groups. For instance, as a consequence of \cite{rozenberg1972}, we show:
	\begin{propositionalpha}
		The marked isomorphism problem is not semi-decidable for groups given by EDT0L presentations. 
	\end{propositionalpha}
	The marked isomorphism problem asks to determine if a given bijection between the generators of two marked groups extends to a group isomorphism. This proposition contrasts with the finite presentation case, where the marked isomorphism problem, while undecidable, is still semi-decidable. 

		\subsection{Marked quotient algorithms}

	One of the successful uses of L-presentations is quotient computation: in  \cite{BARTHOLDI2008}, it is shown that it is possible, given a finite L-presentation of a group $G$ and a natural number $n$, to compute a finite presentation for the nilpotent quotient $G/\gamma_{n}(G)$, where $\gamma_{n}(G)$ designates the $n$-th term in the lower central series of $G$, defined by  $\gamma_{1}(G)=G$ and $\gamma_{n+1}(G)=[\gamma_{n}(G),G]$. The given algorithms were in fact implemented and shown to allow explicit computations that run reasonably fast, especially in the case of ascending L-presentations. In \cite{Hartung2011}, Hartung shows that coset enumeration can be applied to finitely L-presented groups, and gives explicit computations of some finite index subgroups. This result relies crucially on the ability to recognize finite quotients of L-presented groups, in the form of the following result: there is an algorithm that takes as input a finite L-presentation for a group and a finite presentation of a finite group, and determines whether the latter is a marked quotient of the former. In the vocabulary of \cite{Rauzy2022}, finitely L-presented groups have computable finite quotients. \\
		We present here a unified proof of the results of \cite{BARTHOLDI2008} and \cite{Hartung2011}, and extend them to the case of abelian-by-nilpotent, hyperbolic and virtual direct product of hyperbolic groups. What we in fact show is that the hypotheses of Noetherianity and of uniform solvability of the word problem, that were used in the case of finite and nilpotent groups, can be replaced more generally by hypotheses of \emph{equational Noetherianity} and of \emph{solvability of the universal Horn theory}. Solvability of the universal theory of virtual direct products of hyperbolic groups was established in \cite{Ciobanu2020}, building on \cite{Dahmani2010}.

		\begin{theoremalpha}[Quotient computations, extending \cite{BARTHOLDI2008}, \cite{Hartung2011}] \label{Thm:quotient computation intro}
			Let $\mathcal{C}$ be any of: the class of abelian-by-nilpotent groups or the class of virtual direct products of finitely many hyperbolic groups. 
			There is an algorithm that, given an EDT0L presentation for a marked group $(G,S)$ and a finite presentation for a marked group $(H,S')$ in $\mathcal{C}$, decides whether or not $(H,S')$ is a marked quotient of $(G,S)$. 
		\end{theoremalpha}

	Theorem \ref{Thm:quotient computation intro} in particular applies to $\mathcal{C}$ being the class of free groups. In this case, we can be more precise:  
	\begin{theoremalpha}[Residually free image] \label{Thm:residually free image}
		There is an algorithm that, given an EDT0L presentation for a marked group $(G,S)$, produces a finite presentation in the category of residually free groups of the residually free image of $(G,S)$.  
	\end{theoremalpha}
	This simply means that given an EDT0L presentation for a marked group $(G,S)$, it is possible to find a finite presentation for a marked group $(H,S')$ so that $(G,S)$ and  $(H,S')$  have exactly the same marked free quotients. 
	Free quotients of finitely generated groups were studied in depth in relation to  Tarski's problem about the universal theory of free groups. A \emph{limit group} is a group which has the same universal theory as a free group. By a theorem of Sela and Kharlampovich and Myasnikov, limit groups are always finitely presentable \cite{Sela2001,Kharlampovich1998}. Another important theorem states: 
	\begin{theorem}[Sela, \cite{Sela2001}]
		For any finitely generated group $G$, there are finitely many limit groups $\{\Gamma_{i},1\le i\le n\}$ and morphisms $\varphi_{i}:G\rightarrow\Gamma_{i}$ such that any morphism from $G$ to a free group factors through one of the morphisms $\varphi_{i}$.  
	\end{theorem}
	The set of limit groups given in the above theorem is the \emph{first layer of the Makanin-Razborov diagram of $G$} \cite{Sela2001}. 
	As a consequence of a theorem of  Groves and Wilton  \cite{Groves2009} we prove: 
	\begin{theoremalpha}
		There is an algorithm that, given an EDT0L presentation for a group $G$, produces the first layer of its Makanin-Razborov diagram. 
	\end{theoremalpha}
	This result seems to be new even for finitely presented groups. 
	
			\subsection{Separation of classes of groups defined by languages}
			
	One of the classical problems in the theory of formal languages is that of identifying and separating families of languages: proving that different definitions yield the same family of language, or on the contrary establishing that some inclusions between families of languages are strict. (The P=NP problem is of this kind, as are all separation problems in complexity theory.)
	We consider the similar problem for classes of groups defined by presentations. In fact, establishing that $\pi(\text{LAN}_1)\ne \pi (\text{LAN}_2)$ is usually much more difficult than establishing the purely language theoretical result $\text{LAN}_1\ne \text{LAN}_2$. 
	We show why quotient computations permit to establish $\pi (\text{EDT0L})\ne \pi (RE)$. \\
	The results of Theorem \ref{Thm:quotient computation intro} are uniform: for $\mathcal{C}$ one of the classes of abelian-by-nilpotent or of virtual direct product of hyperbolic groups, there is an algorithm which, given as input both an EDT0L presentation for a marked group $(G,S)$ and a finite presentation for a marked group in $\mathcal{C}$, determines whether or not the latter is a marked quotient of the former. 
	We can consider a weaker statement, the non-uniform version of the above result: ``for each group $G$ that admits an EDT0L presentation, there is an algorithm that, given a finite presentation for a marked group in $\mathcal{C}$, determines whether it is a marked quotient of $G$''. The point here is that this non-uniform version of Theorem \ref{Thm:quotient computation intro} provides a possibly non-trivial property of each group that admits an EDT0L presentation. 
	In fact, if we take $\mathcal{C}$ to be the class of abelian-by-$k$-nilpotent groups (for a fixed $k$) or the class of free groups, the obtained property is trivial: for each finitely generated group, there is an algorithm that recognizes its marked abelian-by-$k$-nilpotent quotients, and one that determines its marked free quotients. Thus the interest of  Theorem \ref{Thm:quotient computation intro} in regards to free and abelian-by-$k$-nilpotent quotients is purely in its uniformity. \\
	However, by \cite{Rauzy2022}, the property of  ``having a recursive set of marked finite quotients'' is non-trivial, and transverse to solvability of the word problem. (Note that this property does not depend on a choice of a marking for the considered group.)
	\begin{theorem}[\cite{Rauzy2022}]
		There exists a finitely generated residually finite group $G$ with solvable word problem such that for any finite generating set $S$ of $G$,  no algorithm can, given a finite presentation of a marked finite group, determine whether or not this presentation defines a marked quotient of $(G,S)$. 
	\end{theorem}
	(The above result obviously remains true replacing ``finite groups'' by ``hyperbolic groups''.) In view of Theorem \ref{Thm:quotient computation intro}, the group given in this theorem cannot have a presentation in EDT0L, even though it is recursively presented. Thus we know that the clauses of Theorem \ref{Thm:quotient computation intro} that concern abelian-by-$k$-nilpotent groups and free groups do not provide information on groups with presentations in EDT0L, while the one that concerns hyperbolic groups does. 
	We finally consider nilpotent quotients, without fixing a bound on the nilpotency class. 
	We prove here:
	\begin{theoremalpha}\label{Thm: uncomputable nilp quotients Intro}
		The exists a residually nilpotent finitely generated group $H$ with solvable word problem but uncomputable nilpotent quotients: no algorithm can, on input $n$, produce a finite presentation of $H/\gamma_n (H)$.
	\end{theoremalpha} 
	The group given by this theorem can be supposed to be central-by-metabelian, as it is constructed as a central quotient of Hall's central-by-metabelian group whose center in free abelian on countably many generators \cite{Hall1954}.  This theorem thus shows that the statement about nilpotent groups of Theorem \ref{Thm:quotient computation intro} gives rise to groups that do not have EDT0L presentations.

	\subsection{Open problems}

	Moving forward, several directions of research appear promising. First, establishing for families of languages  other than EDT0L whether they give rise to groups that admit marked quotient algorithms.  
	Secondly, finding methods other than the construction of groups with non-computable quotients to separate classes of groups defined by languages. For instance we ask: 
	\begin{problem}
		Find a family of languages LAN such that $$\pi(\text{EDT0L})\subsetneq \pi(\text{LAN})\subsetneq\pi(\text{RE}).$$
	\end{problem} 
	Finally, we include problems related to the class $\pi(\text{DTF0L})$:
	\begin{problem}
		Which group varieties admit free groups with presentations in  $\pi(\text{DTF0L})$? 
	\end{problem} 
	
	\subsection*{Acknowledgements}
	We would like to thank Constantinos Kofinas for very helpful remarks about endomorphisms of free center-by-metabelian groups. We would also like to thank Martin Kassabov for helpful discussions. Thanks also to Corentin Bodart and Alex Bishop for insights on context-free languages. Finally, we would like to thank Simon André for references and remarks about limit groups. 
	
\section{Preliminaries}

\subsection{Language Theory}

Let $A$ be any set, to which we refer from now on as the \textbf{alphabet}.
The free monoid over $A$, i.e. the set of all words over $A$, including the empty word which we denote by $\epsilon$, is denoted by $A^*$. If not specified, the term morphism always refers to monoid homomorphism.
A \textbf{formal language} over $A$ is a set of words over $A$, i.e. a subset of $A^*$.
A \textbf{family of languages over $A$} is a set of formal languages over $A$. \\
A \textbf{family of languages} is then defined to be a map that associates to a set $A$ a family of languages over $A$. \\
In practice (see \cite{Rozenberg1997}), families of languages come equipped with ways of giving finite descriptions  of the considered languages, that take the form of machines (automata, Turing machines, and so on) or of grammars. We could thus give a more precise definition of  the notion of ``family of languages'', that reflects more precisely how this term is used, thanks to the theory of numberings, see for instance \cite{Weihrauch1987}. But we will not use such a definition here. \\
We will usually denote families of languages in all capitals, such as LAN. 
 \subsubsection{Regular and context free languages}\label{Subsec: Regular and context free}
 Here, we introduce regular and context free languages, denoted REG and CF. Denote by FIN the family of finite languages. In Section \ref{subsec: First Non Injectivity Result}, we show that in fact: $$\pi(\text{FIN})=\pi(\text{REG})=\pi(\text{CF}).$$

\begin{definition}
	A \textbf{deterministic finite automaton (DFA)} over a finite alphabet $A$ is a tuple $(S,s_0,\delta,F)$ where
	\begin{itemize}
		\item $S$ is a finite set referred to as the \textbf{states} of the DFA,
		\item $s_0 \in S$ is a distinguished \textbf{initial state},
		\item $\delta: S \times A \rightarrow S$ is the \textbf{transition function},
		\item $F \subseteq S$ is the set of \textbf{accepting states}.
	\end{itemize}
\end{definition}

A word $a_0\dots a_n \in A^*$ can be processed by such a DFA: Starting in state $s_0$, one computes $s_1=\delta(s_0,a_0)$ and repeatedly $s_{i+1}=\delta(s_i,a_i)$ until $i=n$. If the  resulting state $s_n$ is contained in $F$, we say that the DFA \textbf{accepts} the word.
The set of all words accepted by a given DFA is called its \textbf{language} and denoted $\Lan(S,s_0,\delta,F)$.

A formal language $L \subseteq A^*$ is called \textbf{regular} (or rational) if it is the language of a DFA over $A$. 

\begin{definition}
	A \textbf{context free grammar} over a finite alphabet $A$ is a tuple $(A,N,R, s_0)$ where
	\begin{itemize}
		\item $N$ is a finite alphabet disjoint from $A$ called the set of \textbf{nonterminal characters}, 
		\item $R$ is the set of \textbf{production rules}, a finite subset of $N\times (A\cup N)^*$ ,
		\item $s_0\in N$ is the \textbf{initial symbol}. 
	\end{itemize}
\end{definition}

Let  $\mathcal{G}=(A,N,R, s_0)$ be a context free grammar. 
For words $u_1$ and $u_2$ in $(A\cup N)^*$, we say that $u_1$ \textbf{directly yields} $u_2$ if there is a production rule $(a,w)\in R$ such that $u_1$ can be written as $u_1=xay$ for words $x$ and $y$ of $(A\cup N)^*$, and that $u_2=xwy$. In this case we write $u_1\Rightarrow_{\mathcal{G}}u_2$. Denote by $\Rightarrow^*_{\mathcal{G}}$ the transitive closure of the relation $\Rightarrow_{\mathcal{G}}$. 

\begin{definition}
	Let  $\mathcal{G}=(A,N,R, s_0)$ be a context free grammar. The \textbf{language generated by $\mathcal{G}$} is defined as $$\mathcal{L}(\mathcal{G})=\{w\in A^*, s_0\Rightarrow^*_{\mathcal{G}}w\}.$$
\end{definition}

\subsubsection{Some L-languages}\label{Subsec: Def L languages}

This section will introduce concepts from the theory of parallel rewriting language families, which are often called L-languages or ``Lindenmayer languages''. For the interested reader, a broader account can be found in chapter 5 of \cite{Rozenberg1997}.

\begin{definition}\label{defsys}
Fix an alphabet $A$
	\begin{enumerate}
		\item A \textbf{DT0L system} over $A$ is a word $w_0$ together with a finite set $\Phi$ of morphisms from $A^*$ into $A^*$.
		\item A \textbf{DTF0L system} over $A$ is a finite set $W_0 \subseteq A^*$ of words together with a finite set $\Phi$ of morphisms from $A^*$ into $A^*$.
		\item Fix an additional alphabet $A_E$ disjoint from $A$. An \textbf{EDT0L system} over $A$ with non-terminals $A_E$ is a DT0L system over $A \cup A_E$.
		\item An \textbf{HDT0L system} over $A$ is a DT0L system over some alphabet $B$ together with a morphism $f:B^* \rightarrow A^*$.
	\end{enumerate}
\end{definition}

In the following definition we will write $\Phi(w)$ and $\Phi(W)$ where $\Phi$ is a set of morphisms, $w$ is a word and $W$ a set of words to mean the sets
$$\Phi(w)=\{\phi(w) \mid \phi \in \Phi\}$$
$$\Phi(W)=\bigcup_{w\in W}\Phi(w)$$
Furthermore $\Phi^n$ denotes the $n$-th iterate of applying $\Phi$. In particular, $\Phi^n(w)$ contains the image of $w$ under any chain of $n$ morphisms in $\Phi$.

\begin{definition}\label{deflang}
Suppose $A$ is an alphabet.
	\begin{enumerate}
		\item A DT0L system $(w_0,\Phi)$ over $A$ defines a language 
		$$\Lan(w_0,\Phi)=\bigcup_{n < \omega} \Phi^n(w_0)$$
		\item A DTF0L system $(W_0, \Phi)$ over $A$ defines a language
		$$\Lan(W,\Phi)=\bigcup_{w \in W} \Lan(w,\Phi)$$
		\item An EDT0L system $(w_0,\Phi)$ over $A$ with non-terminals $A_E$ defines a language
		$$\Lan(w_0,\Phi,A_E)=\Lan(w_0,\Phi) \cap A^*$$
		\item A HDT0L system $(w_0,\Phi, f)$ defines a language
		$$\Lan(w_0,\Phi,f)=f(\Lan(w_0,\Phi))$$ 
	\end{enumerate}
\end{definition}

We adopt the common terminology of saying that a formal language is DT0L (DTF0L, EDT0L,...) if there is a corresponding system which defines that language. We also denote by DT0L, DTF0L, ... the respective families of languages. Additionally, FIN denotes the family of all finite languages.

\begin{remark}
	DT0L is strictly contained in DTF0L.
\end{remark}
\begin{proof}
The family DTF0L trivially contains all finite languages. The words of the finite language can be taken as the seed words of the DTF0L system and the morphisms to be the identity.\\
In particular, the finite single character language $\{a, a^2\}$ is DTF0L. However, it is not DT0L: If it was, then either $a$ or $a^2$ must be the seed word. If $a$ is the seed, a morphism mapping $a$ to $a^2$ must be part of the defining DT0L system for the language. But then $a^4$ would ne to be in the language as well. The word $a^2$ cannot be the seed word either because there would have to be a morphism mapping $a^2$ to $a$. Such a morphism does not exist, there is no suitable choice for the image of $a$ under such a morphism.
\end{proof}

The following is a result of Nielsen, Rozenberg, Salomaa and Skyum. It can be found in \cite{Nielsen1974}.

\begin{lemma}[\mbox{\cite[Lem. 3.3, Thm. 6.4, 6.5]{Nielsen1974}}]\label{eqhedt0l}
The families EDT0L and HDT0L are effectively the same, i.e. there is an algorithm that computes from an EDT0L system an HDT0L system that defines the same language and vice versa.
\end{lemma}

The following can be found in \cite{Rozenberg1976}.

\begin{lemma}
The family EDT0L is effectively closed under finite union, applying morphisms and intersection with a regular language.
\end{lemma}

We denote by DTF0L+FIN the family of languages that are a union of a DTF0L and a finite language. 

\begin{proposition}\label{prop:dtf0lpsub}
	DTF0L+FIN is a proper subfamily of EDT0L.
\end{proposition}
\begin{proof}
For being a subfamily it suffices, due to the closure properties of EDT0L, to show that both DTF0L and FIN are subfamilies of EDT0L. This is achieved by two very similar constructions: One defines a single non-terminal $N$ which is also the seed of the system. Additionally, one adds morphisms that map $N$ to each word in the finite language or seed word of the DTF0L system respectively and that act trivially everywhere else. In the DTF0L case, the morphisms of the original system are added as well, acting trivially on $N$.\\
The following example proves that the inclusion is proper: Consider the maximal language $\Lan=\{a^n\mid n \in \N\}$ over the single letter alphabet $A=\{a\}$. It is easily be seen to be EDT0L: The system consists of a single non-terminal $N$ as the seed word and two morphisms mapping $N \mapsto Na$ and $N \mapsto \epsilon$.
However, the language is not DTF0L+FIN. If it was, the defining system would have seed words $a^{s_1},\dots,a^{s_k}$, finitely many morphisms defined by $a \mapsto a^{k_i}$ and a finite language $\{a^{n_1},\dots,a^{n_l}\}$. There is a number $m$ that is not a product of any of theses numbers (for example choose a large enough prime). Any word in the language defined by this system is of the form $a^{s_iK}$ where $K$ is a finite product of the $s_i$, or it is one of the $a^{n_i}$. Hence $a^m$ is not in the language defined by the system.
\end{proof}

\subsubsection{Regular Control}

In \cite{Asveld1977} Asveld proved that regular control does not increase the power of various language models, among them the EDT0L formalism. His tools have hence been adapted and used by various authors, for example by Ciobanu, Elder and others \cite{Ciobanu2021},\cite{Ciobanu2023},\cite{Ciobanu2015} to construct certain EDT0L languages.

We summarize this method as follows.

\begin{definition}\label{def:regcedt0l}
	A \textbf{regularly controlled EDT0L system} is an EDT0L system $(w_0,\Phi, A_E)$ together with a regular language $\mathcal{R}$ over the alphabet $\Phi$. As each word in $\mathcal{R}$ defines a composition of morphisms\footnote{Note that this correspondence uses a left-to-right order of composition, i.e. the morphism which is the first letter of a word is also applied first. This can be thought of in the same spirit as morphisms acting on the right on their domain.} and can therefore itself be seen as a morphism, one defines the language of a regularly controlled EDT0L system by
	$$\Lan(w_0,\Phi,\mathcal{R})=\mathcal{R}(w_0) \cap A^*.$$
\end{definition}

It is important to note, although the regular control $\mathcal{R}$ can be an infinite set, it is still defined by finite data, as all regular languages are. Hence also regularly controlled EDT0L systems have a finite description.
This fact is rather unsurprising in light of the next result which is a corollary of Theorem 2.1 in \cite{Asveld1977}.

\begin{lemma}
	The family of languages defined by regularly controlled EDT0L systems is effectively identical to the family of EDT0L languages.
\end{lemma}

In fact, in \cite{Ciobanu2015} and \cite{Diekert2017} the authors give an even slightly more restrictive version of Definition \ref{def:regcedt0l}, which is still equivalent.

\begin{lemma}
	If one removes from Definition \ref{def:regcedt0l} the intersection with terminal words $A^*$, so in fact requires the regular control $\mathcal{R}$ to consist only of morphisms mapping $w_0$ to $A^*$, the defined family of languages is still the same.
\end{lemma}
\begin{proof}
Suppose $(w_0,\Phi,A_E,\mathcal{R})$ is a regularly controlled EDT0L system. We claim that the set 
$$\mathcal{R}'=\{ f \in \Phi^* \mid f(w_0) \in A^*\}$$
is regular over $\Phi$, where again we identify a word over $\Phi$ with the morphisms that is the composition of its letters. Assuming this claim, one notices that $\mathcal{R} \cap \mathcal{R}'$ is still a regular set as an intersection of regular sets and 
$$\mathcal{L}(w_0,\Phi,A_E,\mathcal{R})=\mathcal{L}(w_0,\Phi,A_E,\mathcal{R}\cap\mathcal{R}')=(\mathcal{R}\cap \mathcal{R}')(w_0).$$
To prove the claim we construct an explicit DFA $(S, s_0,\delta, F)$ over the alphabet $\Phi$ that accepts $\mathcal{R}'$. We denote by $\mathfrak{l}:(A \cup A_E)^* \rightarrow \mathcal{P}(A\cup A_E)$ the mapping of a word to the set of the letters used in it.\\
Then we can define the automaton by $S=\mathcal{P}(A \cup A_E)$, $s_0=\mathfrak{l}(w_0)$, $F=\mathcal{P}(A)$ and $\delta(M,\phi)=\mathfrak{l}(\phi(M))$ for any $M \in S$ and $\phi \in \Phi$. The automaton tracks for a composition of morphisms which letters are used as part of its image of $w_0$. It exactly accepts those compositions that only use letters from $A$, so they map to $A^*$. As $A$ and $A_E$ are finite, so is their power set, and we have defined a DFA accepting $\mathcal{R}'$.
\end{proof}

\subsection{Presentations of groups in families of languages}

\subsubsection{Marked groups and presentations}

Let $k$ be a natural number. A\textbf{ $k$-marked group} is a finitely
generated group $G$ together with a $k$-tuple $S=(s_{1},...,s_{k})$
of elements of $G$ that generate it. Note that repetitions are allowed in $S$, the order of the
elements matters, and $S$ could contain the identity element of $G$.
A \textbf{morphism of marked groups} between $k$-marked groups $(G,(s_{1},...,s_{k}))$
and $(H,(t_{1},...,t_{k}))$ is a group morphism $\varphi$ between
$G$ and $H$ that additionally satisfies $\varphi(s_{i})=t_{i}$, $i\in \{1,..,k\}$.
It is an \textbf{isomorphism of marked groups} if $\varphi$ is a group
isomorphism. Marked groups are considered up to isomorphism.
\\
Note that there is at most one morphism of marked groups from a marked group $(G,S)$ to a marked group $(H,T)$. If there are morphisms $(G,S)\rightarrow (H,T)$  and $(H,T)\rightarrow (G,S)$, then the marked groups $(H,T)$ and $(G,S)$ are isomorphic. 
Thus isomorphism classes of marked groups define a partially ordered set, which can easily be seen to be a lattice. 
If $G$ is a group, we say that any marked group of the form $(G,S)$  is a \textbf{marking} of $G$. We precise \textbf{$k$-marking} if $S$ is a $k$-tuple. 
\\
For each $k\in \N$ we fix a free group $\mathbb{F}_k$ of rank $k$ together with a basis $S_k$.
A $k$-marking of a group $G$ can then
be seen as an epimorphism $\varphi:\mathbb{F}_{k}\rightarrow G$,
the image of $S_k$ by $\varphi$ defines a marking with respect to
the previous definition. Two $k$-marked groups are then isomorphic
if they are defined by morphisms with identical kernels: the isomorphism
classes of $k$-marked groups are classified by the normal subgroups
of a rank $k$ free group. \\
If $G$ is a group and $A$ a subset of $G$, we denote by $\langle\langle A\rangle  \rangle^{G}$  the normal closure of $A$ in $G$. We omit the exponent $G$ when there is no ambiguity about which is the ambient group. 
If $S$ is a set, we denote by $\mathbb{F}_S$ the free group over $S$. 
\begin{definition}
A \textbf{presentation} is a pair $\langle S\vert R \rangle$, where $S$ is a set and $R$ is a subset  of the set of words  over $S\cup S^{-1}$. 
The marked group defined by the presentation   $\langle S\vert R \rangle$ is the pair 
$$(\mathbb{F}_S/\langle \langle R\rangle \rangle, S) .$$ 
\end{definition}

The elements of the set $R$ above are usually referred to as the ``relators'' of the presentation. It is common to specify them in terms of equations. For example, the equation $ab=cd$ in a presentation has the same meaning as the word $abd^{-1}c^{-1}$.

\begin{definition}
Let $(G,S)$ be a marked group and LAN a family of languages. We say that $(G,S)$ has \textbf{a LAN-presentation} if $G$ has a group presentation
$$(G,S)=\langle S|R \rangle$$
such that $R$ is a language in LAN over the alphabet $S \cup S^{-1}$.\\
\end{definition}

\begin{definition}
	If LAN is a family of languages, define:
	$$\pi_M(\text{LAN}): \text{ the set of marked groups with a presentation in LAN};$$
	$$\pi(\text{LAN}): \text{ the set of groups with some marking in }\pi_M(\text{LAN}).$$
\end{definition}

Note that an important point which is not apparent in these definitions is the following: in practice, families of languages come equipped with a type of finite description for the languages they contain (usually in the form of a grammar). 
And thus the definition above also allows to define a type of finite description for marked groups: a tuple $S$ of symbols for generators, together with the finite description of a language over $S \cup S^{-1}$, will constitute a finite description of a marked group. 
It is these finite descriptions that permit us to ask decidability questions about groups given by presentations.  \\
This could be made more precise thanks to the use of \emph{numberings}, see \cite{Weihrauch1987}. A numbering of a set $X$ is a partial map which goes from a subset of $\N$ to $X$. This formalizes the fact that elements of $X$ have finite descriptions. The maps $\pi$ and $\pi_M$ are functions that associate to a numbering of a set of languages  numberings of  abstract or marked groups. 
We do not give any more details here because we will not use the possibilities offered by making this formalism explicit.

The concept of L-presentation, or endomorphic presentation, was introduced in \cite{Bartholdi2003}. 
\begin{definition}
	A \textbf{finite L-presentation} is a quadruple 
	$$\langle S\mid Q\mid \Phi\mid R\rangle $$
	which consists in a set $S$ of generators, finite sets $Q,R \subseteq \mathbb{F}_S$ of reduced words which are the \textbf{additional} and \textbf{initial} relations, and a finite set $\Phi$ of group endomorphisms $\mathbb{F}_S \rightarrow \mathbb{F}_S$. This defines a marked group
	$$\langle S\mid Q\mid \Phi\mid R\rangle=(\mathbb{F}_S/ \langle \langle Q \cup \bigcup_{n\in\N }\Phi^n (R)\rangle\rangle,S).$$
	
\end{definition}
An L-presentation is called \textbf{ascending} if the set $Q$ in the above definition is empty. 

\subsubsection{First examples for the non-injectivity of $\pi$ and $\pi_M$ }\label{subsec: First Non Injectivity Result}

Some of the smallest families of languages usually considered are finite and regular languages, families denoted FIN and REG. 
On the other hand of the spectrum, the biggest families usually considered are recursive (REC) and recursively enumerable (RE) languages. 
In each case these different families of languages yield the same classes of (marked)  groups.

The following are folklore results:
\begin{proposition}[Craig's trick]
	$\pi_M(\text{REC})=\pi_M(\text{RE})$, and thus 	$\pi(\text{REC})=\pi(\text{RE})$. 
\end{proposition}
\begin{proposition}
	$\pi_M(\text{FIN})=\pi_M(\text{REG})$ and $\pi(\text{FIN})=\pi(\text{REG})$.
\end{proposition}

Here we add the following result. 
\begin{proposition}\label{prop: Context free= finite}
	Let CF be the family of context free languages. Then  $\pi_M(\text{CF})=\pi_M(\text{FIN})$
\end{proposition}
We will use the classical pumping lemma for context free languages. 

\begin{lemma}[Pumping lemma for context free languages]\label{lem: pumping cfl}
	If $L$ is a context free language over an alphabet $A$, there exists a natural number $p$ such that  any element $a$ of $L$ of length at least $p$ can be split into five words $a=uvwxy$, such that 
	\begin{enumerate}
		\item The words $v$ and $x$ are not both empty,
		\item The word $vwx$ has length at most $p$, 
		\item For any $n\ge 0$, the word $uv^n w x^ny$ also belongs to $L$.
	\end{enumerate}
\end{lemma}

\begin{proof}[Proof of Proposition \ref{prop: Context free= finite}]
Let $(G,S)=\langle S|L\rangle$ be a marked group given by a presentation with a context free language of relators $L$. Fix the constant $p$ for the language $L$ given by the pumping lemma. We are going to prove that
$$(G,S) = \langle S\mid \langle \langle L \rangle\rangle \cap \{w \in \mathbb{F}_S \mid |w| \le 2p\}\rangle$$
where $\langle\langle L \rangle\rangle$ denotes the normal closure of $L$ in the free group over $S$.
\\
 Thus we want  to show that the relations of $(G,S)$ of size at most $2p$ imply all the relations of $L$. Suppose $r \in L$ with $|r| > 2p$. By the pumping lemma we can write $r$ as
$$r=uvwxy$$
with $|vwx| \le p$ and $uwy \in L$. This implies that the equation
$$uvwxy=1=uwy$$
hold in $(G,S)$, and therefore so does $vwx=w$. Note that $|vwxw^{-1}|\le|vwx|+|w|\le 2p$, and thus  $vwx=w$ is a relation of $(G,S)$ of size at most $2p$. Notice also that having the equations
$$uwy=1 \quad \text{ and } \quad vwx=w$$
implies $uvwxy=1$ by replacing $w$ in $uwy$ by $vwx$. Since $|uwy|<|uvwxy|$, one can apply the above process inductively until $|uwy|<2p$. \\
And thus the relations of $(G,S)$ of size at most $2p$ are sufficient to defining it.
\end{proof}

A similar phenomenon occurs in the comparison of DTF0L and EDT0L: $$\pi(\text{DTF0L+FIN})=\pi(\text{EDT0L}).$$ However, we conjecture that $\pi_M(\text{DTF0L+FIN})\ne \pi_M(\text{EDT0L})$. This appears in Section \ref{subsec: identification Lpres EDT0L}.

\subsubsection{Tietze Transformations }
Here, we give, thanks to Tietze transformations, natural conditions on a family of languages LAN so that the set of marked groups $\pi_M (\text{LAN})$ be saturated for the abstract group isomorphism relation. The conditions are given in Lemma \ref{lem: change of generators}.  

A Tietze transformation of a presentation $\langle S|R \rangle$ is one of the following four operations: 
\begin{enumerate}
	\item Adding a redundant relator: if $w$ belongs to  $\langle\langle R\rangle\rangle^{\mathbb{F}_S}$, add $w$ as a relator.  
	\item  Removing a redundant relator:  if $w$ belongs to $R$, and   $\langle\langle R\rangle\rangle^{\mathbb{F}_S}=\langle\langle R\setminus \{ w \}\rangle\rangle^{\mathbb{F}_S}$, remove  $w$ from $R$.
	\item Adding a redundant generator: add a new generating symbol $x\notin S $ along with a new relation $x=w$, where $w$ is a word over $S\cup S^{-1}$.  
	\item Removing a redundant generator:  If $s$ belongs to $S$, and if there is a relation in $R$ of the form $s=w$, where  $w$ is a word over $S\cup S^{-1} \setminus \{s,s^{-1}\}$, then remove $s$ from $S$, and replace all occurrences of $s$ in relations of $R$ by the word $w$.  
\end{enumerate}

It is obvious that applying Tietze transformations to a presentation does not change the group it defines.

Notice that Tietze transformations are of two types: those that induce a change of marking (3.+4.) and those that change the presentation without changing the marking (1.+2.). \\
 A classical lemma is then (see \cite{Lyndon1977}):
\begin{lemma}
	Any two presentations of the same group can be transformed one into the other by Tietze transformations. 
	Any two finite presentations of the same group can be transformed one into the other by finitely many Tietze transformations. 
\end{lemma}

The following procedure describes how to transform a presentation of a group $G$ on a generating set $S$ to a presentation of $G$ on a generating set $T$. 

\begin{itemize}
	\item Start with a presentation  $\langle S|R \rangle$ for the marked group  $(G,S)$. 
	\item Fix for each element $t$ of $T$ a word $w_t\in (S\cup S^{-1})^*$ that expresses it in terms of the elements of $S$. 
	\item By Tietze transformations, $\langle S, T|R, t=w_t,\, t\in T \rangle$ is now another presentation of $G$. 
	\item Because  $T$ is a generating set, for each element $s$ in $S$ there is a word $w_s$ in $(T \cup T^{-1})^*$ that expresses it in terms of the generators of $T$.
	\item We transform the presentation of $G$ into  $\langle S, T|R, t=w_t,\, t\in T , s=w_s, \, s\in S \rangle$ by adding redundant relators. 
	\item In the obtained presentation, the generators of $s$ are redundant, since they are defined in terms of the generators of $T$.
	\item Apply Tietze transformations to delete the elements of  $S$ from the generator list.
	\item The obtained presentation is $$ \langle T|R(w_s), t=w_t (w_s), t\in T \rangle,$$ where by $R(w_s)$ we mean ``the relations of $R$ where all occurrences of elements of $S$   have been replaced by their expressions in terms of the elements of $T$''. 

\end{itemize}

A consequence of the above construction is the following lemma: 
\begin{lemma}\label{lem: change of generators}
	If  LAN is a family of languages stable under taking homomorphic images and taking union with finite sets,  then the set of marked groups that have a presentation in LAN is closed under group isomorphism: if a group has a presentation in LAN with respect to a marking, it has one with respect to any marking. 
\end{lemma}
\begin{proof}
This follows directly from the steps described  above. 
\end{proof}
This lemma applies to EDT0L. The following is an effective version of the above lemma, which also applies to EDT0L.  
	
\begin{lemma} \label{lem: effective change of generators}
	Suppose that LAN is a family of uniformly recursively enumerable  languages. Suppose furthermore that LAN is effectively stable under taking homomorphic images and taking union with finite sets. Then there is a procedure that takes as input a  LAN-presentation of a marked groups $(G,S)$, and outputs a sequence of  LAN-presentations of $G$ that contains at least one LAN-presentation  for each marking of $G$. 
\end{lemma}

\begin{proof}
All finite generating sets 	of $(G,S)$ can be enumerated. Indeed, given a finite set $T$ of words on $S\cup S^{-1}$, we can prove that they form a generating family of  $(G,S)$  by finding expressions of the elements of $S$ in terms of the elements of $T$: expressions of the form 
$$s=w_s,\, w_s\in (T\cup T^{-1})^*,$$
for each $s\in S$. Such expressions can be found because equality is semi-decidable in recursively presented groups.
\\
We can then form a presentation of the form 
$$\langle S,T \vert R, t=w_t,\, t\in T , s=w_s, \, s\in S \rangle.$$
We can then apply the Tietze transformations that delete the generators $S$.  What remains is the presentation  $$ \langle T|R(w_s), t=w_t (w_s), t\in T \rangle,$$the hypotheses on the family LAN of languages immediately imply that this presentation is again a LAN-presentation. 
\end{proof}

\subsubsection{Semi-decidability of the isomorphism problem for finitely presented groups}

The \textbf{isomorphism problem} is the problem of deciding whether two given groups are isomorphic. This makes sense in any class of groups attached to a type of finite descriptions for its elements. 
The \textbf{marked isomorphism problem} is  the problem of deciding whether two given marked groups are isomorphic as marked groups. 
A problem is \textbf{semi-decidable} if there is a Turing machine that stops exactly on positive instances of the problem. A problem is \textbf{co-semi-decidable} when there is a Turing machine that stops exactly on negative instances of the problem. 
A problem is \textbf{decidable} when there is a Turing machine that solves it. A well known lemma says that this is exactly when this problem is semi-decidable and co-semi-decidable. \\
A problem which is undecidable (i.e. not decidable) can thus still be semi-decidable or co-semi-decidable, or it can be neither of them.
\\Note that we will also use the notion of \textbf{recursively enumerable set}: a set is recursively enumerable when there is a Turing machine that enumerates its elements. When equality is a semi-decidable relation, being recursively  enumerable implies being semi-decidable, but these two notions differ in general.

\begin{lemma}\label{lem: marked iso problem for finitely presented groups}
	The marked isomorphism problem is semi-decidable for finitely presented groups.
\end{lemma}
\begin{proof}
	Two finite presentations $\langle  S \vert R \rangle$ and $\langle  S \vert U \rangle$ over the same set of generating symbols $S$ define the same marked group if and only if the following  inclusions hold:  $$R \subseteq  \langle \langle U\rangle \rangle^{\mathbb{F}_S},$$ $$U \subseteq  \langle \langle R\rangle \rangle^{\mathbb{F}_S}.$$
	These conditions are semi-decidable, because in a free group, the normal subgroup generated by a finite set is recursively enumerable.
 \end{proof}

\begin{lemma}
	The isomorphism problem is semi-decidable for finitely presented groups. 
\end{lemma}
\begin{proof}
This is a joint application of Lemma \ref{lem: effective change of generators} and Lemma \ref{lem: marked iso problem for finitely presented groups}. 
\end{proof}

The results above about finitely presented groups are well known (they were first noticed by Mostowski \cite{Mostowski1973}). However, they are presented here in a way that emphasizes two distinct aspects of the isomorphism problem: the marked isomorphism problem and the problem of enumerating all the markings of a given group. 
This is relevant to the present study, because only the ``change of marking'' half of  this theorem generalizes to EDT0L presentations. \\
The following problem generalizes the isomorphism problem. 

\begin{definition}
	Let $(G,S)$ be a marked group and  $\mathcal{C}$ be  a class of marked groups associated to a type of finite descriptions. The \textbf{marked quotient problem for $(G,S)$ and $\mathcal{C}$ } is the problem, given a marked group in $\mathcal{C}$, to decide whether or not it is a marked quotient of $(G,S)$.\\
	If $X$ is a class of groups attached with a type of finite description, then we can talk about the \textbf{uniform marked quotient problem for $X$ and $\mathcal{C}$}, which is the problem, given a marked group in $X$ and one in $\mathcal{C}$, to decide whether the latter is a marked quotient of the former. 
\end{definition}

\begin{proposition}
	The finitely presented marked quotient problem is uniformly semi-decidable for finitely presented groups.  
\end{proposition}
\begin{proof}
We have seen above that the marked isomorphism problem is semi-decidable for finitely presented groups. Given two finite presentations $\langle S \vert R\rangle$ and $\langle S' \vert R'\rangle$, deciding whether  $\langle S' \vert R'\rangle$ is a marked quotient of $\langle S \vert R\rangle$ is equivalent to deciding whether $\langle S \vert R, \, R' (S)\rangle$ is another presentation for the same marked group as  $\langle S \vert R\rangle$. Thus the marked quotient problem reduces to the marked  isomorphism problem.
\end{proof}

\subsection{Computing quotients and dismissing infinite presentations}
When working with infinite presentations of groups, a natural question is the following: Is it possible, given an infinite presentation of a finitely presentable group, to computably point out a finite subset of the infinite presentation that is sufficient to define this group?  \\
Let LAN be a family of languages, and $\mathcal{C}$ a set of finitely presented groups. 
In this section, we give conditions under which the following two problems are equivalent: 
\begin{enumerate}
	\item Given a presentation in LAN of a group in $\mathcal{C}$, find a finite presentation for it. 
	\item Given a presentation in LAN for a marked group $(G,S)$, and a finite presentation for a group $(H,S')$ in $\mathcal{C}$, determine whether $(H,S')$ is a marked quotient of $(G,S)$. 
\end{enumerate}

\begin{lemma}\label{lem: marked quotient and finite pres}
	Let $\mathcal{C}$ be  a recursively enumerable class of finitely presented groups. Let LAN be a family of uniformly recursively enumerable languages.
	If the marked quotient problem in  $\mathcal{C}$ is uniformly semi-decidable for groups  given by presentations in LAN, then there is an algorithm which, given as input a presentation in LAN for a marked group in $\mathcal{C}$, produces a finite presentation for this group. 
\end{lemma}
\begin{proof}
	Given a presentation in LAN for a marked group $(G,S)$, enumerate all finite presentations of groups in  $\mathcal{C}$. Use semi-decidability of the quotient problem to enumerate all marked quotient of  $(G,S)$. For each marked quotient $(H,T)$ of  $(G,S)$ obtained this way, we will try to prove that  $(G,S)$ is in fact also a quotient of  $(H,T)$, which will show that they are isomorphic as marked groups. 
\\
But to prove that $(G,S)$ is a marked quotient of  $(H,T)$, it suffices to check that the finitely many defining  relations of $(H,T)$ hold in $(G,S)$. This is semi-decidable as soon as $(G,S)$ was given by a recursive presentation.
\end{proof}

An open problem is to find a family LAN of formal languages for which the above lemma can be applied with $\mathcal{C}$ being the set of all finitely presented groups. 
It is an open problem to decide whether groups given by EDT0L presentations have uniformly semi-decidable marked quotient problem for finitely presented groups. 
	\\

\begin{lemma}\label{lem: iso P EDT0L equiv to quotient problem}
	Let $\mathcal{C}$ be a class of finitely presented groups. Suppose that the isomorphism problem is solvable in  $\mathcal{C}$ for groups given by finite presentations, and that  $\mathcal{C}$  is recursively enumerable via finite presentations. Let LAN be family of languages. If the marked quotient problem in  $\mathcal{C}$ is uniformly semi-decidable for groups  given by presentations in LAN, then the isomorphism problem is solvable for groups in  $\mathcal{C}$ given by presentations in LAN. 
\end{lemma} 

\begin{proof}
This follows directly from Lemma \ref{lem: marked quotient and finite pres}. 
\end{proof}
This can be applied with $\text{LAN}=\text{EDT0L}$ and with $\mathcal{C}$ being any of the sets of finite groups, of nilpotent groups, of hyperbolic groups. \\
The above lemma admits a converse. 

\begin{lemma}\label{lem: compute finite pres => quotients are computable}
	Let $\mathcal{C}$ be  a class of finitely presented groups stable under taking quotients. Let LAN be a family of languages for which the operation ``taking union with finite sets'' is computable. 
	Then, if the marked isomorphism problem is solvable for groups in $\mathcal{C}$ with presentations in LAN, the $\mathcal{C}$-marked quotient problem is uniformly semi-decidable for groups given by LAN presentations.
\end{lemma} 
 
\begin{proof}
Just notice that the group given by $\langle S_1 \vert R_1\rangle$ is a marked quotient of  $\langle S_0 \vert R_0\rangle$ if and only if $\langle S_0 \vert R_0\rangle$ and $\langle S_0 \vert R_0, R_1(S_0)\rangle$ define the same marked group. The hypothesis that $\mathcal{C}$ is quotient stable guarantees that   $\langle S_0 \vert R_0, R_1(S_0)\rangle$ belongs to $\mathcal{C}$ in any case.
 \end{proof}

It seems that all  classes of finitely presented groups for which the isomorphism problem is known to be solvable also happen to give semi-decidable properties of groups. Examples of this include hyperbolic groups (\cite{Papasoglu1995} for semi-decidability, \cite{Dahmani2011} for  the isomorphism problem), virtually nilpotent groups (\cite{Grunewald1980} for the isomorphism problem, semi-decidability is easy) and limit groups (\cite{Bumagin2007} and \cite{Dahmani2008} for the isomorphism problem and \cite{Groves2009} for semi-decidability). Of course, somewhat artificial examples of classes with solvable isomorphism problem that are not semi-decidable are obtained by taking subsets of the above (there are uncountably many sets of hyperbolic groups, some of these are necessarily not recursively enumerable). 
The above lemmas thus apply to all the ``natural'' classes quoted above. 

\section{Identifying EDT0L presentations and L-presentations}

\subsection{Identification result}\label{subsec: identification Lpres EDT0L}

The main results of this section is the following:
 \begin{theorem}\label{thm: EDT0L=Lpres DTF0L=Ascending L pres}
 	Let $G$ be a finitely generated group. The following are equivalent: 
 	\begin{enumerate}
 		\item the group $G$ admits an L-presentation;
 		\item the group $G$ admits a presentation in DTF0L+FIN;
 		\item the group $G$ admits a presentation in EDT0L. 
 	\end{enumerate} 
 	 \end{theorem}

\begin{proof}

From the definitions it is obvious that L-presentations produce their relators exactly by the DTF0L+FIN mechanism, with an additional constraint: the iterated morphisms should be group endomorphisms, and not simply monoid endomorphisms.
Thus having an L-presentation implies having a DTF0L+FIN presentation. 
And of course, since DTF0L+FIN is a proper subfamily of EDT0L, having a presentation in DTF0L+FIN implies having a presentation in EDT0L. 
\\
We now prove the last implication: that the class of EDT0L presented groups is  contained in the class of L-presented groups.\\
	The proof in fact relies on the HDT0L characterisation of EDT0L languages given by Nielsen, Rozenberg, Salomaa and Skyum, see Lemma \ref{eqhedt0l}. 
	Suppose $(w_0,\Phi,f)$ is a HDT0L system on alphabets $S,T$, i.e. $f$ is a morphism from $S^*$ to $T^*$ and $(w_0,\Phi)$ is a DT0L system on $S$.
We have to  construct an L-presentation $\langle S'|R|\Phi'|Q\rangle$ of the group $\langle T|\Lan(w_0,\Phi,f)\rangle$. 
This is achieved by choosing $S'=S \sqcup T$ and $R = \{w_0\}$. Furthermore, $\Phi'$ is the set of the extensions of the morphisms of $\Phi$ to $S'$ by acting as the identity on $T$. Finally, 
$$Q = \{f(s)=s\mid s \in S\}$$
Remember that $f(s)$ in the above definition is a word in $T^*$. We will now prove that the L-presentation $\langle S \cup T|w_0|\Phi'|Q\rangle$ defines the same group as the HDT0L presentation $\langle T|\Lan(w_0,\Phi,f)\rangle$.
By the definition of L-presentations, the L-presentation above defines the (possibly infinite) classical group presentation $\langle S \cup T \mid M\rangle$ where $M$ is the set of relators
$$M= \{\Phi^k(w_0) \mid k \ge 0\} \cup \{f(s)=s\mid s \in S\}$$
The first part is exactly the language of the DT0L system $(w_0, \Phi)$, so
$$M=\Lan(w_0,\Phi) \cup \{f(s)=s \mid s \in S\}$$
For each $s$ in $S$, the relation $s=f(s)$ gives an expression of the generator $s$ as a product of the generators of $T$: this shows that this generator is redundant, and it can be deleted by a Tietze transformation. All occurrences of $s$ in the presentation are then replaced by $f(s)$. 

We then obtain a new presentation $\langle T\mid M'\rangle$ for the same group, where $M'$ is the language $\Lan(w_0,\Phi)$ with every occurrence of a letter $s \in S$ replaced by $f(s)$. This is exactly the language  $f(\Lan(w_0,\Phi))$. So $M'=f(\Lan(w_0,\Phi))=\Lan(w_0,\Phi,f)$.
\end{proof}

We finally want to remark a subtlety that one might overlook: it is not clear whether or not $\pi(\text{DTF0L})$ defines exactly the set of  groups that admit ascending L-presentations. 
Indeed,  the definitions of these two classes of groups are almost identical, with the slight difference that in one case one iterates free group endomorphisms, whereas in the other case free monoid endomorphisms are iterated. 
Indeed, a group given by an ascending L-presentation can be written as 

$$\mathbb{F}_S/ \langle \langle \bigcup_{\phi\in \Phi ^* }\phi (R)\rangle \rangle,$$ where $R$ is a finite set of relations, $\Phi$ is a finite set of group endomorphisms of  $\mathbb{F}_S$, and $\Phi ^*$ designates the monoid generated by $\Phi$. We obtain the notion of presentation in DTF0L if instead of supposing that the elements of $\Phi$ are group endomorphisms of  $\mathbb{F}_S$, we only ask that they be monoid endomorphisms of $(S\cup S^{-1})^*$. (Thus for instance a transformation given by $a\mapsto a$ and $a^{-1}\mapsto a^2$ is a valid monoid endomorphism but not a free group endomorphism.) \\
This distinction need not arise in the EDT0L case, since the monoid endomorphisms acting on non-terminals can directly be seen as group endomorphisms.

\subsection{Application to the marked isomorphism problem}

One of the purposes of relating the study of L-presentations to the theory of formal languages is to be able to apply results from formal language theory. Here, we do so with a result of Rozenberg \cite{rozenberg1972}: equality of EDT0L languages is not decidable. 

Because EDT0L presentations generalize finite presentations, the isomorphism problem is undecidable for groups given by EDT0L presentations. 
However, the isomorphism problem for finitely presented groups is semi-decidable: there is an algorithm that, on input two finite presentations, will halt exactly when they define the same group.
We have seen above that this fact can be seen as the conjunction of two results: the marked isomorphism problem is semi-decidable for finitely presented groups, and there is an algorithm that, on input a finite presentation for a marked group, enumerates a list of finite presentations that will contain at least one presentation for  each marking of this group. 
This second algorithm exists for groups with EDT0L presentations, and more generally whenever the conditions of Lemma \ref{lem: effective change of generators} are met. \\
We show now that the marked isomorphism problem for groups given by  EDT0L presentations is strictly harder than for groups given by  finite presentations. 
\begin{proposition}
	The marked isomorphism problem is not semi-decidable for groups given by EDT0L presentations.  
\end{proposition}
\begin{proof}
We use the fact that equality of EDT0L languages is not semi-decidable, see  \cite{rozenberg1972}. (Note that the   theorem  stated in  \cite{rozenberg1972}  is  ``equality is undecidable for EDT0L languages''. But for every family of uniformly recursive languages the equality problem is co-semi-decidable, since to prove that two languages are different it suffices to produce an element which belongs to one but not to the other. Hence the statement that equality is not decidable shows actually that equality is not semi-decidable for EDT0L languages.) 
\\
Let $L$ be an EDT0L language over an alphabet $S$. Consider a copy $\hat{S}$ of $S$, and two letters $a$ and $b$ that do not belong to $S$. 
For a word $w$ of $S^*$, denote by $\hat{w}$ the same word rewritten over the alphabet $\hat{S}$. \\
Consider the marked group given by the presentation 
$$\pi (L)=\langle S, \hat{S}, a, b \, \vert \,  waw=(\hat{w}b\hat{w})^{-1}, \, w\in L\rangle.$$
The proposition  follows directly from the following facts:
\begin{itemize}
	\item The presentation $\pi (L)$ is given by an EDT0L language, and an EDT0L grammar for it can be computed from the EDT0L grammar of $L$.
	\item We have the equivalence  $$\pi (L)=\pi (L')\iff L=L',$$ where the first equality corresponds to  marked group isomorphism. 
\end{itemize}
We now prove those two facts. 
To produce an EDT0L grammar for the language $\{waw\hat{w}b\hat{w},\, w\in L\}$, we proceed as follows. If $N$ denotes the set of non-terminals for $L$, we consider a copy $\hat{N}$ of $N$ as a new set of non-terminals. Add also the two letters $a$ and $b$ as new terminals. 
If $w_0$ was the seed for the grammar of $L$, replace it by the new seed $$w_0 a w_0 \hat{w_0}b\hat{w_0}.$$ Finally, extend all morphisms that appear in the grammar for $L$ in the following way. If $\phi $ is a morphism in the grammar of $L$, it is defined on  $S\cup N$. We  extend it to  $\hat{S}\cup \hat{N}$ by the formula    $$\phi (\hat{u})= \hat{\phi (u)}$$ for $u$ in $S\cup N$, and by $\phi(a)=a$ and $\phi(b)=b$.
\\ It is easy to see that this grammar produces exactly the desired language. 
\\
We now prove the second point: the marked group defined by the presentation  $\pi (L)$  determines uniquely the language $L$. 
This follows directly from the normal form theorem for free products with amalgamation, which gives us the following equivalence: in the marked group given by  $\pi (L)$,  $$waw=(\hat{w}b\hat{w})^{-1} \iff w\in L.$$
To show that the presentation  $\pi (L)$ indeed defines an amalgamated free product, it suffices to show that the families $waw$  and $(\hat{w}b\hat{w})^{-1}$, for $w\in L$, generate isomorphic marked groups. 

But the given families are both bases of the free group they generate. This follows from the fact that they are Nielsen reduced families \cite{Lyndon1977}: given two distinct words $w$ and $w'$ in $L$, the two occurrences of the letter $a$ that appear in a product  $(waw)^{\pm 1}(w'aw')^{\pm 1}$  can never cancel. 
\end{proof}

The marked quotient problem is the problem of deciding, given two marked groups, whether the first one is a marked quotient of the second one. 
\begin{corollary}
	The marked quotient problem in EDT0L is not semi-decidable for groups given by EDT0L presentations. 
\end{corollary}

\begin{proof}
	Two marked groups are equal if and only if each one is a marked quotient of the other. 
\end{proof}
One could expect much stronger results than the ones above. It would be very interesting to find examples of a family of languages LAN where the equality and inclusion problems are semi-decidable, while the isomorphism and quotient problem are not semi-decidable for groups with presentations in LAN.\\
A particular class of groups that admit EDT0L presentations for which semi-decidability of the isomorphism problem is an open problem is those consisting of finitely many relations and identities. 

\begin{problem}
	Show that equality is not semi-decidable for groups given by finitely many relations and identities. 
\end{problem}
\begin{problem}
		Find an example of an identity $W(x)=1$ such that the problem ``given a finite presentation, does the group it defines satisfy $W$?'' is not semi-decidable.
\end{problem}
\begin{problem}[\cite{MillerIII1992}]
	Is being metabelian semi-decidable for groups given by finite presentations? 
\end{problem}
The best results in this direction were obtained by Kleiman \cite{Kleiman1983}: 

\begin{theorem} [\cite{Kleiman1983}]
	There exists a finitely generated group $G$ with solvable word problem in which the identity problem is not semi-decidable: no algorithm can, on input an element of a free group, seen as  defining an identity, stop if and only if this identity is satisfied in $G$. 
\end{theorem}
It would be very interesting to obtain the same result, but with the additional assumption that $G$ has an EDT0L presentation (or, better, a finite presentation).
Another interesting problem related to group varieties is the following: 
\begin{problem}
	Which group varieties admit free groups that have DTF0L presentations? 
\end{problem}
Results from the study of endomorphism of free groups in group varieties are useful in regards to this problem, as \cite{Kofinas2022}. 
Note that the problem of determining which varieties have finitely presented free groups is very hard. It is for instance an open problem to tell for which $p\in\mathbb{N}$ does the variety defined by $x^p=1$ has finitely presented free groups (this is related to the Burnside problem).

\section{Unified approach to computing marked quotients}

In this section we give a unified approach to computing marked quotients of EDT0L presented groups, thereby unifying and generalizing  \cite{BARTHOLDI2008} and \cite{Hartung2011}. (Note also that \cite{BARTHOLDI2008} and \cite{Hartung2011} deal with L-presentations, and while this does not make a difference in terms of groups, it possibly does in terms of marked groups, and the following results deal with marked groups.)
The following theorems were known:
\begin{theorem}[Nilpotent quotients, \cite{BARTHOLDI2008}] \label{thm:nilpquotsection}
	There is an algorithm that, given an EDT0L presentation for a marked group $(G,S)$ and a positive integer $k$, computes a finite presentation for the marked group $(G/\gamma_k(G),S)$.
\end{theorem}

\begin{theorem}[Finite quotients, \cite{Hartung2011}]\label{thm:finquotsection}
	There is an algorithm that, on input an EDT0L presentation for a marked group $(G,S)$, a finite presentation for a marked finite group $(H,S')$, and a function $f:S\rightarrow S'$, decides whether or not $f$ can be extended to a group homomorphism. 
\end{theorem}
We add:
\begin{theorem}[Abelian-by-$k$-nilpotent quotients] \label{thm:Abelianbynilpquot}
	There is an algorithm that, given an EDT0L presentation for a marked group $(G,S)$, and an integer $k$, produces a finite presentation in the category of abelian-by-k-nilpotent groups of $(G/[\gamma_{k}(G),\gamma_{k}(G)],S)$.
\end{theorem}
What we mean in the above is that $(G/[\gamma_{k}(G),\gamma_{k}(G)],S)$ is given by a finite presentation together with the identity  of abelian-by-k-nilpotent groups as an additional infinite set of relations.

\begin{theorem}[Residually free image]\label{thm:freequotsection}
	There is an algorithm that, on input an EDT0L presentation for a marked group $(G,S)$, produces a relative finite presentation in the category of residually free groups for the residually free image of $(G,S)$.
\end{theorem}
\begin{theorem}[Quotients that are virtual direct products of hyperbolic groups]\label{thm:VirtualHyperbolicquotients}
	There is an algorithm that, on input an EDT0L presentation for a marked group $(G,S)$, a finite presentation for a marked  group $(H,S')$ which is virtually a direct product of hyperbolic groups, and a function $f:S\rightarrow S'$, decides whether or not $f$ can be extended to a group homomorphism. 
\end{theorem}

The theorems above in each case rely on two properties : 
\begin{itemize}
	\item For nilpotent and finite groups (the cases that were known), on noetherianity and uniform solvability of the word problem.  
	\item For abelian-by-nilpotent groups, noetherianity is replaced by the Max-n property \cite{Hall1954}, which is noetherianity for normal subgroups. We also use uniform solvability of the word problem. 
	\item For direct products of hyperbolic groups, we rely on equational noetherianity instead of actual noetherianity, and on decidability of the universal theory instead of the word problem (\cite{Ciobanu2020}, extending the result of \cite[Corollary 9.1]{Dahmani2010} on hyperbolic groups). (Decidability of the universal Horn theory would suffice.)
\end{itemize}

\subsection{Subgroup functors of fully characteristic subgroups }

The notion of subgroup functor is used in the study of finite groups. See for instance \cite{BallesterBolinches2016} and references therein.  
\begin{definition}
	Let $\mathcal{C}$ be a class of groups. Let $\kappa$ be a function that associates to each group $G$ in $\mathcal{C}$ a set $\kappa(G)$ of subgroups of $G$. We say that $\kappa$ is a \textbf{subgroup functor} for $\mathcal{C}$ if it respects the isomorphism relation:  if $f:G_1\rightarrow G_2 $ is an isomorphism, then $\kappa(G_2)=f(\kappa (G_1))$. 	
\end{definition}

We will consider subgroups functors with additional properties. 
\begin{definition}\label{def: Good subgroup functor}
	A \textbf{residual} is a subgroup functor $\kappa$ such that:
\begin{enumerate}
	\item For each $G$ in  $\mathcal{C}$, $\kappa (G)$ is a singleton. 
	\item For any two groups $G_1$ and $G_2$ in $\mathcal{C}$ and any group homomorphism $\phi : G_1 \rightarrow G_2$ the image of $\kappa (G_1)$ by $\phi$ lies in  $\kappa (G_2)$. 
		$$\forall G_1, G_2 \in \mathcal{C}, \forall \phi :G_1 \rightarrow G_2, \, \phi(\kappa (G_1))\le \kappa(G_2).$$
	\item  For all $G$,  $\kappa (G/\kappa(G))=\{1\}$.
\end{enumerate}
\end{definition}
In the definition above, applying condition 2 to $G_1=G_2$, we see that for each $G_1$, $\kappa (G_1)$ is fully characteristic in $G_1$.\\
As a consequence of the third condition, we see that  for any group $G$, $$(G/\kappa(G)) \big / \kappa(G/\kappa(G)) =G/\kappa(G).$$ 
\begin{example}
Let $\mathcal{Z}(G)$ designate the center of $G$. Then $\mathcal{Z}$ satisfies only the first condition of Definition \ref{def: Good subgroup functor}. \\
Let  $\mathcal{T}(G)$ designates the subgroup of $G$ generated by its elements of finite order.  Then $\mathcal{T}$ satisfies the first and second conditions of Definition \ref{def: Good subgroup functor}, but not the third. 
\end{example}

Residuals are not unrelated the notion of  ``residually-$\mathcal{C}$" group, where $\mathcal{C}$ is a class of groups. 

\begin{lemma}
	Suppose that $\kappa$ is a residual. Then there is a class $\mathcal{C}$ of groups such that for all $G$, $$\kappa(G)=\bigcap_{ H\in \mathcal{C},\, \phi:G:\rightarrow H} \ker \phi $$
	And thus $G/\kappa(G)$ is the biggest residually-$\mathcal{C}$ quotient of $G$. 
\end{lemma}

\begin{proof}
We just take $\mathcal{C}$ to be $\{G, \kappa(G)=\{1\}\}$. Then:
	\begin{align*}
		w\not \in \kappa(G)&\implies  w\ne 1 \text{ in $G/\kappa (G)$}\\
		&\implies  w \notin \bigcap_{ H\in \mathcal{C},\, \phi:G:\rightarrow H} \ker \phi
\end{align*}
And: 
\begin{align*}
	w \in \kappa(G)&\implies  \forall H\in\mathcal{C}, \forall \phi:G\rightarrow H, \phi(w)=1\\
	&\implies  w \in \bigcap_{ H\in \mathcal{C},\, \phi:G:\rightarrow H} \ker \phi
\end{align*}
\end{proof}

Note also the following corollary, which expresses the fact that the commutation property $$\forall G, \forall R\subseteq G, \, G/\kappa(G)/\langle\langle R\rangle\rangle=G/\langle\langle R\rangle\rangle/\kappa(G/\langle\langle R\rangle\rangle)$$  characterizes group varieties.
\begin{corollary}
	Suppose that $\kappa$ is a residual, and that the property $\kappa (G)=\{1\}$ is quotient stable. Then the set of groups $G$ with  $\kappa(G)=\{1\}$ is a group variety, and for any group $\kappa(G)$ is the subgroup of $G$ generated by the verbal subgroups associated to this variety. (It is thus itself a verbal subgroup when the variety is defined by finitely many identities.)
\end{corollary}
\begin{proof}
	This simply follows from the fact that a group variety is a set of groups stable by unrestricted direct product, taking subgroups and taking quotients, by a result of  Birkhoff \cite{Birkhoff1935}. Because the set of residually $\mathcal{C}$-groups, for any $\mathcal{C}$, is always stable by taking direct products and subgroups, quotient stability is the only missing hypothesis. 
\end{proof}

For any group $G$ we will now abbreviate $G/\kappa (G)$ by $G/\kappa$. Also, for $R$ a set of elements of a group  $G$, we denote by $G/R$ the quotient $G/\langle\langle R\rangle \rangle$, omitting to express the normal closure.  
\begin{lemma}\label{lem: three marked group morphisms }
	Suppose that $\kappa$ satisfies conditions 1 and 2 in Definition \ref{def: Good subgroup functor}. 
	Let $(G,S)$ be a marked group. Let $R$ be set of relations. 
	 Then there are morphisms of marked groups:
	$$(G/R /\kappa, S)   \rightarrow (G/\kappa/R/\kappa,S)\rightarrow (G/R/\kappa/\kappa,S). $$
\end{lemma}
\begin{proof}
	This follows directly from the second condition in Definition \ref{def: Good subgroup functor}. 
\end{proof}

\begin{lemma}\label{lem: commutation mod kappa and add relations}
 Suppose that $\kappa$ is a residual. For any group we have that  
	$$((G/\kappa )/R) /\kappa =(G/R) /\kappa.$$
\end{lemma}

\begin{proof}
	This follows directly from the previous lemma. 
\end{proof}

\begin{lemma}\label{lem: deleting relations modulo quotient by characteristic subgroup}
Let $\kappa$ be a residual. Let $(G,S)$ be a marked group. Let $R$ be a set of relations. 
Suppose that the relations of $R$ hold in $(G/\kappa,S) $.  Then there is an  isomorphism of marked groups:
$$((G/R )/\kappa, S)   =  (G/\kappa,S).$$
\end{lemma}
\begin{proof}
This follows directly from  the previous lemma. 
\end{proof}

\subsection{First Lemma : Detecting stabilization on non-terminals is sufficient}

Fix a marked group $(G,S)$ given by an HDT0L presentation. We thus consider an alphabet of non-terminals $A$, a finite set $\Phi$ of endomorphisms of $A^*$, a seed word $w_0\in A^*$, and a morphism $f:A^*\rightarrow S^*$.  
The language we are interested in is: 
$$f(\bigcup_{i=0}^\infty \Phi ^i (w_0)).$$

\begin{lemma}[Stabilization on non-terminals]\label{lem: stabilizaton on non-terminals suffices}
	Let $\kappa$ denote a residual. 
	Let $n$ be a natural number.   If   $$\langle A \, \vert \,\bigcup_{i=0}^n \Phi ^i (w_0) \rangle/\kappa =\langle A \, \vert \, \bigcup_{i=0}^\infty \Phi ^i (w_0) \rangle/\kappa,$$  then 
	$$ \langle S \, \vert \, f(\bigcup_{i=0}^n \Phi ^i (w_0)) \rangle/\kappa   \text=   \langle S \, \vert \, f(\bigcup_{i=0}^\infty \Phi ^i (w_0))\rangle/\kappa . $$ (Here, equality designates marked group isomorphism.)
	
\end{lemma}

\begin{proof}
	We have the following morphisms:
	 $$\langle A\vert \bigcup_{i=0}^n \Phi ^i (w_0) \rangle \overset{f}{\rightarrow} \langle S \vert f (\bigcup_{i=0}^n \Phi ^i (w_0) )\rangle\overset{\pi}{\rightarrow}  \langle S \vert f (\bigcup_{i=0}^n \Phi ^i (w_0) )\rangle/\kappa.$$
	 Because $\kappa$ defines a residual, the composition $\pi \circ f$ factors through  $\langle A\vert \bigcup_{i=0}^n \Phi ^i (w_0) \rangle/\kappa$. 
	 But by hypothesis 
	 $$\langle A\vert \bigcup_{i=0}^n \Phi ^i (w_0) \rangle/\kappa=\langle A\vert \bigcup_{i=0}^\infty \Phi ^i (w_0) \rangle/\kappa.$$ 
	 Thus all relations of $f (\bigcup_{i=0}^\infty \Phi ^i (w_0) )$ are already satisfied in $\langle S \vert f (\bigcup_{i=0}^n \Phi ^i (w_0) )\rangle/\kappa$. 
	 This is summed up in the following commutative diagram: 
	  \begin{center}
	 	\begin{tikzcd}
	 		\langle A\vert \bigcup_{i=0}^n \Phi ^i (w_0) \rangle  \arrow[d, "f"] \arrow[rd, "\pi"] & 
	 		
	 		\\ \langle S \vert f (\bigcup_{i=0}^n \Phi ^i (w_0) )\rangle \arrow[d, "\pi"] & 	\langle A\vert \bigcup_{i=0}^n \Phi ^i (w_0) \rangle /\kappa=\langle A\vert \bigcup_{i=0}^\infty \Phi ^i (w_0) \rangle /\kappa  \arrow[ld, "\exists g"]
	 		
	 		\\ \langle S \vert f (\bigcup_{i=0}^n \Phi ^i (w_0) )\rangle/\kappa  & \end{tikzcd}
	 \end{center}
	 	We conclude by applying  Lemma \ref{lem: deleting relations modulo quotient by characteristic subgroup}. 
\end{proof}

\subsection{Second  Lemma: Effective Noetherianity}

In our applications, we will use the following residuals $\kappa$:
\vspace{0.3cm}
\\
\centerline {
	\bgroup
	\def\arraystretch{1.5}
	\begin{tabular}{|c|c|}
		\hline 
	Nilpotent Quotients &	Abelian-by-Nilpotent Quotients   \\
		\hline
		$\kappa(G)=\gamma_{k}(G)$, for a given $k$ in $\N$ &	$\kappa(G)=[\gamma_{k}(G),\gamma_{k}(G)]$, for a given $k$ in $\N$ \\
		\hline
		Residually free  image &Virtual direct products of hyperbolic groups    \\
		\hline	\begin{tabular}{c} $\kappa(G)=\bigcap_{\phi:G\rightarrow \mathbb{F}_2}\ker (\phi)$    \end{tabular}	&\begin{tabular}{c} $\kappa(G)=\bigcap_{\phi:G\rightarrow H}\ker (\phi),$  \\ where $H$ is the given group \end{tabular} \\
		\hline
	\end{tabular}
	\egroup
}
\vspace{0.3cm}

We include two effective noetherianity lemmas: one deals with abelian-by-nilpotent and finite quotients, and the other one with free and hyperbolic quotients.
The reason for this is that the stabilization that occurs is different in both cases: in the first two cases, when we quotient a free group by the considered characteristic subgroup, we get a group which satisfies the Max-n property: all increasing sequences of normal subgroups stabilize. And thus the sequence 
$$	\mathbb{F}_A/\kappa /\langle\langle\bigcup_{i=0}^1 \Phi ^i (w_0) \rangle\rangle \rightarrow  	\mathbb{F}_A/\kappa /\langle\langle\bigcup_{i=0}^2 \Phi ^i (w_0) \rangle\rangle\rightarrow  	\mathbb{F}_A/\kappa /\langle\langle\bigcup_{i=0}^3 \Phi ^i (w_0) \rangle\rangle\rightarrow ...$$ eventually stabilizes. 
On the other hand, when $G/\kappa(G)$ is the residually free image of $G$, or its residually-$H$ image for a hyperbolic group $H$,  it does not satisfy the Max-n property in general, and the sequence that stabilizes is one where adding new relations is intertwined with taking residually free images (resp. residually-$H$ images). The sequence that stabilizes is then: 
$$	\mathbb{F}_A /\langle \langle w_0 \rangle \rangle/\kappa \rightarrow  	\mathbb{F}_A /\langle\langle \Phi ^1 (w_0) \rangle \rangle/\kappa/\langle\langle \Phi ^2 (w_0) \rangle \rangle /\kappa \rightarrow   ...$$

\begin{lemma}[Effective noetherianity for abelian-by-nilpotent and finite quotients]\label{lem: eff noeth Nilp Meta Finite}
	Suppose that $\kappa$ is one of the subgroup functors associated to abelian-by-nilpotent or finite quotients as presented in the table above. There exists an algorithm that on input a word $w_0$ and  a set $\Phi$ of endomorphisms of a free group $\mathbb{F}_A$ produces a natural number $n$ such that 
	$$\langle\langle \bigcup_{i=0}^n \Phi ^i (w_0) \rangle \rangle^{\mathbb{F}_A/\kappa (\mathbb{F}_A)} =\langle \langle \bigcup_{i=0}^\infty \Phi ^i (w_0) \rangle\rangle ^{\mathbb{F}_A/\kappa (\mathbb{F}_A)}.$$
	
\end{lemma}
\begin{proof}
	This follows directly from the fact that in each case $\mathbb{F}_A/\kappa$ satisfies the Max-n property. It is a theorem of Hall \cite{Hall1954} for abelian-by-nilpotent groups. For the subroup functor $\kappa$ associated to computing finite quotients, it is trivial, since in this case the group $\mathbb{F}_A/\kappa (\mathbb{F}_A)$ is finite. \\
	Thus there exists $n$ such that the equality claimed in the lemma occurs. \\
	This equality is equivalent to the fact that each $\phi$  in $\Phi$ already defines a group homomorphism in the quotient $(\mathbb{F}_A/\kappa (\mathbb{F}_A))/\langle\langle \bigcup_{i=0}^n \Phi ^i (w_0) \rangle\rangle.$  It is possible to check algorithmically that this is the case: it suffices to check that each of the finitely many elements of $ \Phi ^{n+1} (w_0)$ is already the identity in 
	$(\mathbb{F}_A/\kappa (\mathbb{F}_A))/\langle\langle \bigcup_{i=0}^n \Phi ^i (w_0) \rangle\rangle.$ This is decidable because the word problem is uniformly solvable for the considered families of groups: abelian-by-nilpotent groups \cite{Baumslag1981} and finite groups. 
	\end{proof}
	
	We will now establish the corresponding result for free quotients and quotients that are virtually direct products of hyperbolic groups.
		A group $G$ is called \textbf{equationally Noetherian} if any infinite system of equations over $G$ is equivalent to a finite subsystem.
	 What we will use to compute free quotients is the following theorem of Guba: 
	
		\begin{theorem}[Guba, \cite{Guba1986}]\label{Thm: Guba }
	 Finitely generated free groups are equationally Noetherian.
	\end{theorem}
	This was later generalized to:
		\begin{theorem}[Sela \cite{Sela2009} in the torsion free case, Weidmann and Reinfeldt \cite{Weidmann2020} for the general case]\label{Thm: Hyperbolic => equationally noeth}
		Gromov hyperbolic groups are equationally Noetherian.
	\end{theorem}
	Note also that being equationally Noetherian is stable under finite direct products (this is immediate), and is also inherited by finite extensions by \cite{Baumslag1997}. Thus groups that are virtually direct products of hyperbolic groups are equationally Noetherian. 
	We will use the following: 
	\begin{proposition}[Houcine, \cite{Houcine2007}]
		If $H$ is an equationally Noetherian  group, every residually-$H$ group is \textbf{finitely presented as a residually-$H$ group.}
	\end{proposition}
	Let $H$ be an equationally Noetherian. Denote by $\kappa$ the residual associated to ``residually-$H$ images'': $\kappa(G)=\bigcap_{\phi:G\rightarrow H}\ker (\phi)$ . An equivalent formulation of the above proposition is: 
	\begin{proposition}\label{Cor: finite pres as Residually free group }
 For any marked group $(G,S)$ given by a presentation $\langle S \vert R\rangle$, there exists a finite subset $T$ of $R$ such that 
 $$\langle S \vert R\rangle/\kappa = \langle S \vert T\rangle/\kappa $$
	\end{proposition}
		\begin{proof}
		The set $R$ can be seen as a system of equations with no parameters, the variables being elements of $S$. As there are no parameters, we can consider that this system is a system of equations in the  group $H$. By equational noetherianity, there is a finite subsystem of equations equivalent to this one, denote it $T$. 
		This shows exactly that a finitely generated subgroup of $H$ is a marked quotient of  $(G,S)$ if and only if it is a marked quotient of $\langle S\vert T\rangle$, and the desired equality follows. 
	\end{proof}
	Denote by $\kappa(G)$ the set of elements of $G$ that are mapped to the identity by any morphism from $G$ to a free group. 
	\begin{lemma}[Effective noetherianity for computing residually free images]\label{lem: noeth for free quotients}
		There exists an algorithm that on input a word $w_0$ and  a set $\Phi$ of endomorphisms of a free group $\mathbb{F}_A$ produces a natural number $n$ such that 	
		$$\langle A\vert \bigcup_{i=0}^n \Phi ^i (w_0) \rangle /\kappa = \langle A \vert \bigcup_{i=0}^\infty  \Phi ^i (w_0) \rangle/\kappa.$$
	\end{lemma}
	
	\begin{proof}
	The abstract existence of $n$ such as in the lemma is given by  Proposition \ref{Cor: finite pres as Residually free group }. 
	We thus have to justify that we can detect that for a certain $n$ the equality 	$$\langle A\vert \bigcup_{i=0}^n \Phi ^i (w_0) \rangle /\kappa = \langle A \vert \bigcup_{i=0}^\infty  \Phi ^i (w_0) \rangle/\kappa$$ is attained. This equality is in fact equivalent to the fact that in the group $\langle A\vert \bigcup_{i=0}^n \Phi ^i (w_0) \rangle /\kappa $, the free group endomorphisms in $\Phi$ already define group endomorphisms of the quotient.  
	But we have the equivalence: 
\begin{enumerate}
	\item Each morphism in $\Phi$ defines a group endomorphism of  $\langle A\vert \bigcup_{i=0}^n \Phi ^i (w_0) \rangle /\kappa $;
	\item The finitely many relations  of $\Phi^{n+1}(w_0)$ already hold in  $\langle A\vert \bigcup_{i=0}^n \Phi ^i (w_0) \rangle /\kappa$.
\end{enumerate}
Indeed, the first point obviously implies the second one. Conversely, suppose that the relations of  $\Phi^{n+1}(w_0)$  hold in  $\langle A\vert \bigcup_{i=0}^n \Phi ^i (w_0) \rangle /\kappa$. 
Each morphism $\phi$ in  $\Phi$  gives a group morphism 
$$\langle A\vert \bigcup_{i=0}^n \Phi ^i (w_0) \rangle  \overset{\phi}{\rightarrow} \langle A \vert \bigcup_{i=0}^{n+1}  \Phi ^i (w_0) \rangle.$$
By the properties of $\kappa$, this morphism descends: 
$$\langle A\vert \bigcup_{i=0}^n \Phi ^i (w_0) \rangle /\kappa  \overset{\phi}{\rightarrow} \langle A \vert \bigcup_{i=0}^{n+1}  \Phi ^i (w_0) \rangle /\kappa.$$
And as by hypothesis we have that $\langle A \vert \bigcup_{i=0}^{n}  \Phi ^i (w_0) \rangle /\kappa=\langle A \vert \bigcup_{i=0}^{n+1}  \Phi ^i (w_0) \rangle /\kappa$, the equivalence is proved. 

Finally, all that is left to see is that the second item above can be checked: we can decide whether the finitely many relations  of $\Phi^{n+1}(w_0)$ already hold in  $\langle A\vert \bigcup_{i=0}^n \Phi ^i (w_0) \rangle /\kappa$. 
This is a direct consequence of Makanin's Theorem \cite{Makanin1985}, which states that the universal theory of free groups is decidable. Indeed, the elements of $\bigcup_{i=0}^{n+1}\Phi^i(w_0)$ are words over the alphabet $A$. Seeing $A$ as a set of variables, each relation $s$ of $\bigcup_{i=0}^{n+1}\Phi^i(w_0)$ defines an equation $s=1$.  
The universal sentence we must decide is then: 
$$\forall A, (\forall s\in \bigcup_{i=0}^{n+1}  \Phi ^i (w_0), s=1)\iff (\forall s\in \bigcup_{i=0}^{n}  \Phi ^i (w_0), s=1).$$
Note that the quantifications over $s$ are both finite, and in fact correspond to finite conjunctions.
This question can be decided by Makanin's algorithm \cite{Makanin1985}. 
	\end{proof}

		\begin{lemma}[Effective noetherianity for computing marked quotients that are virtual direct products of hyperbolic groups]\label{lem: noeth for virtual hyperbolic quotients}
		Let $H$ be a finitely generated group with a finite index subgroup which is a direct product of hyperbolic groups, and $\kappa$ the associated residual. 
		There exists an algorithm that on input a word $w_0$ and  a set $\Phi$ of endomorphisms of a free group $\mathbb{F}_A$ produces a natural number $n$ such that 	
		$$\langle A\vert \bigcup_{i=0}^n \Phi ^i (w_0) \rangle /\kappa = \langle A \vert \bigcup_{i=0}^\infty  \Phi ^i (w_0) \rangle/\kappa.$$
	\end{lemma}
	\begin{proof}
		The proof is identical to the case of free groups: we use the fact that $H$ is equationally Noetherian (as hyperbolic groups are by \cite{Weidmann2020}, and as equational noetherianity is stable by direct product and finite extension \cite{Baumslag1997}), and that its universal theory is decidable \cite{Ciobanu2020} (and this uniformly in a given finite presentation of $H$).  
	\end{proof}
	
	\subsection{End of the proofs}
We 	simply put together the previous lemmas. 

\begin{proof}[Proofs for Theorems \ref{thm:nilpquotsection}, \ref{thm:finquotsection}, \ref{thm:Abelianbynilpquot}, \ref{thm:freequotsection}, \ref{thm:VirtualHyperbolicquotients}.]
	This is simply an application of Lemma \ref{lem: stabilizaton on non-terminals suffices} together with Lemma \ref{lem: eff noeth Nilp Meta Finite} for finite and abelian-by-nilpotent  quotients,  with Lemma \ref{lem: noeth for free quotients} for the residually free image, and with Lemma \ref{lem: noeth for virtual hyperbolic quotients} for virtually hyperbolic quotients. 
\end{proof}

\subsection{Computation of the first layer of the Makanin-Razborov diagram}

	Recall the following theorem of Sela: 
\begin{theorem}[Sela, \cite{Sela2001}]
	For any finitely generated group $G$, there are finitely many limit groups $\{\Gamma_{i},1\le i\le n\}$ and morphisms $\varphi_{i}:G\rightarrow\Gamma_{i}$ such that any morphism from $G$ to a free group factors through one of the morphisms $\varphi_{i}$.  
\end{theorem}
These limit groups constitute \emph{first layer of the Makanin-Razborov diagram of $G$}. \\
The theorem thus stated is a priori not constructive. And non-constructive arguments are involved in classical proofs (see for instance \cite{Champetier2005} where the above theorem is proved by Zorn's lemma and a compactness argument).
However, we will show that as a consequence of a theorem of Groves and Wilton, the limit groups that appear in the theorem above can actually be recovered from an EDT0L presentation of $G$. 
\begin{theorem}[Groves, Wilton, \cite{Groves2009}]\label{Thm: GrovesWilton}
	There is an algorithm that lists all finite presentations that define limit groups. 
\end{theorem}

Here, we will deduce from Theorem  \ref{thm:freequotsection} the following result: 

\begin{theorem}\label{Thm: MR diagram}
	There is an algorithm which, given an EDT0L presentation for a marked group $(G,S)$, produces the first layer of its Makanin-Razborov diagram. 
\end{theorem}
Note that this result seems to be new even if we replace ``EDT0L presentation" by ``finite presentation" in its statement. 
This restricted theorem, that deals only with finite presentations, is in fact really a reformulation of the original theorem of Groves and Wilton \cite{Groves2009}, since it is easily seen to imply it. Indeed, applying the algorithm given in the theorem above to the sequence of all finite presentations of groups gives an enumeration of all limit groups. 

\begin{proof}[Proof of Theorem \ref{Thm: MR diagram}]
	First apply Theorem \ref{thm:freequotsection} to compute a finite presentation $\langle S \vert R \rangle$ of a group which has exactly the same marked free quotients as $(G,S)$. \\
	By \cite{Groves2009}, it is possible to enumerate all finite presentations of limit groups, and thus also all finite sets of finite presentations of limit groups. Amongst these finite sets, one gives the desired first layer of the Makanin-Razborov diagram (it could be the empty set if the group has no free quotients). 
Thus all we have to show is that it is possible to verify that a given finite set of limit groups is the desired set. But this follows immediately from Makanin's theorem which states that the universal theory of free groups is decidable. Indeed, consider a finite set $\{\Gamma_{i},1\le i\le n\}$ of limit groups given by presentations $\langle S_i\vert R_i \rangle$. We can identify $S$ with the generating families $S_i$ since it suffices to work with marked groups morphism that are formal bijections on the generating families. The set of limit groups defined above does constitute the first layer of the Makanin-Razborov diagram of the group given by the presentation $\langle S \vert R \rangle$ exactly when the following  universal sentence is satisfied by all free groups: $$\forall S, \, \bigwedge_{r\in R} r=1 \iff \bigvee_{1\le i \le n} \bigwedge_{s\in R_i} s=1  $$
Remains in fact the special case of the empty set: to decide whether or not the group has an empty Makanin-Razborov diagram (i.e. no free quotients),  one needs to check whether the following holds in all free groups: $$\forall S, \, \bigwedge_{r\in R} r=1.$$
(As before, in the universal sentences above, the relations are seen as words whose variables are the generators of $S$.)
\end{proof}

\section{A residually nilpotent group with solvable word problem but non computable lower central series }
In this section we prove:
\begin{theorem}\label{Thm: uncomputable nilp quotients}
The exists a residually nilpotent finitely generated group $H$ with solvable word problem but uncomputable nilpotent quotients: no algorithm can, on input $n$, produce a finite presentation of $H/\gamma_n (H)$. This group can be supposed central-by-metabelian. 
\end{theorem}
The same construction also yields:
\begin{theorem}\label{Thm: non effectively residually nilp}
	There exists a residually nilpotent finitely generated group $H$ with solvable word problem which is not \textbf{effectively residually nilpotent}: there does not exist an algorithm that, on input a non-identity element $w$ of  $H$, produces a finite presentation of a nilpotent quotient of $H$ in which $w$ has a non-trivial image (even though this necessarily exists). In fact, no algorithm can, on input a word $w$, determine $n$ such that $w$ has a non-trivial image in $H/\gamma_{n}(H)$. 
\end{theorem}
In view of Theorem  \ref{Thm:quotient computation intro}, we immediately get:
\begin{corollary}
	The group given in Theorem \ref{Thm: uncomputable nilp quotients} cannot have an EDT0L presentation. 
\end{corollary}
And thus Theorem \ref{Thm: uncomputable nilp quotients}, together with the results of \cite{Rauzy2022}, gives a complete description the parts  of  Theorem  \ref{Thm:quotient computation intro} that give trivial or non-trivial properties:
\begin{enumerate}
	\item ``Having computable marked abelian-by-k-nilpotent quotients'' (for a fixed $k$) or ``computable  marked free quotients'' are trivial properties satisfied by all finitely generated groups. 
	\item ``Having computable marked finite quotients'' or ``computable marked nilpotent quotients'' are non-trivial properties that not all group need to satisfy (even amongst residually finite  (resp. nilpotent) groups with solvable word problem), but that groups in $\pi(\text{EDT0L})$ do satisfy. 
\end{enumerate}

\subsection{Convenient presentation of Hall's 3-solvable group }

We will use a group introduced by P. Hall in \cite{Hall1954}. We first describe it. 
In what follows,  we use the convention that $[ x,y ]=x^{-1} y^{-1} x y$, and define inductively $[x,y,z]=[[x,y],z]$, and so on. Denote $x^y=y^{-1}xy$. \\
Hall's group is generated by two elements $a$ and $b$. For $i$ in $\mathbb{Z}$, denote by $b_i$ the element $a^{-i}ba^{i}$. The group  $G$ is then given by the presentation:
$$\pi_1=\langle a,b \, \vert \, [b_i,b_j,a]=1, \,[b_i,b_j,b]=1, \, \forall i ,j \in \mathbb{Z} \rangle $$
Note that $[b_i,b_j]^{a}=[b_{i+1},b_{j+1}]$, so $[b_i,b_j,a]=1\iff [b_i,b_j]=[b_{i+1},b_{j+1}]$, and thus the above presentation could be written equivalently: 
$$\pi_1'=\langle a,b \, \vert \, [b,b_i,a]=1, \,[b,b_i,b]=1, \, \forall i  \in \mathbb{Z} \rangle $$
The group $G$ can both be seen as a maximal central extension of the wreath product $\mathbb{Z} \wr \mathbb{Z}$ and as an  HNN  extension of an infinitely generated two step nilpotent group  generated by the elements $b_i$, $i\in\mathbb{Z}$.
The group $G$ is thus center-by-metabelian and (class 2 Nilpotent)-by-abelian. 
Note that the presentation above is EDT0L.  

Denote by $d_i$ the element $[b_0,b_i]$, for $i\in\mathbb{Z}$. Then $d_i=d_{-i}^{-1}$ and $d_i=[b_t,b_{t+i}]$ for all $t\in\mathbb{Z}$.

The centre of  $G$ is a free abelian group on countably many generators \cite{Hall1954},  in fact the elements $d_i$, $i >0$, form the basis of the centre of $G$.

We first describe another presentation of $G$ that allows to compute its nilpotent quotients easily. 

\begin{lemma}\label{lem: presentation Hall Gp}
	The marked  group $(G,(a,b))$ is also given by the presentation 
	$$\pi_2=\langle a,b \, \vert \, \forall n \in \N,\, [b,\underbrace{a,...,a}_n,b,a]=1,\, [b,\underbrace{a,...,a}_n,b,b]=1\rangle$$ 
\end{lemma}

\begin{proof}
	We use again the notation $b_i=a^{-i} b a^i$ and $d_i=[b_0,b_i]$. \\Recall that, in any group, and for any elements $x,y,z$, we have:
	$$[x z, y] = [x, y]^z \cdot [z, y];$$ $$\left[x^{-1}, y\right] = [y, x]^{x^{-1}}.$$
	 \\
	Denote by $f_n$ the element $[b,\underbrace{a,...,a}_n,b]$ of $G$ (for $n\ge 1$). 
	\\
	Consider first $f_1=[b,a,b]$. We have  $[b,a]=b_0^{-1}b_1$, and thus
	$f_1=[b_0^{-1}b_1,b_0]=[b_1,b_0]=d_1^{-1}$. 
	
	And thus the two relations of $\pi_1$ that force  $d_1=[b_0,b_1]$  to be central are equivalent to the two relations of $\pi_2$ which impose that $f_1$ is central. 
	\\
	We show by induction that the same holds for each $n$: the relations that impose  that $f_i$ is central, $i\le n$, are equivalent to the relations that impose that $[b_i,b_j]$ is central, for $\vert i-j\vert \le n$. 
	We also prove via the same induction: for each $n$ there is an equality (which follows both from the relations of  $\pi_1$ and $\pi_2$): 
	 $$[b,\underbrace{a,...,a}_n]=b_0^{(-1)^n\tbinom{n}{n}}b_{1}^{(-1)^{n-1}\tbinom{n}{n-1}}...b_{n-1}^{(-1)^1\tbinom{n}{1}}b_n^{(-1)^0 \tbinom{n}{0}} C$$
	where $C$ designates a central product of commutators $[b_i,b_j]$ with $\vert i-j\vert <n$.
	\\
	The base case of this induction was proven above: $[b,a]=b_0^{-1}b_1$, and enforcing  $f_1$ to be central is equivalent to imposing  that $[b_0,b_1]$ should be central. 
	\\ 
	Induction step. We compute  $[b,\underbrace{a,...,a}_{n+1}]=[b,\underbrace{a,...,a}_n]^{-1}[b,\underbrace{a,...,a}_n]^a$.
	We replace   $[b,\underbrace{a,...,a}_n]$ by $b_0^{(-1)^n\tbinom{n}{n}}...b_0^{(-1)^0 \tbinom{n}{0}}C$ in that expression. 
	Conjugation by $a$  shifts every index: 
	$$[b,\underbrace{a,...,a}_n]^{a}=b_{1}^{(-1)^n\tbinom{n}{n}}b_{2}^{(-1)^{n-1}\tbinom{n}{n-1}}...b_n^{(-1)^1\tbinom{n}{1}}b_{n+1}^{(-1)^0 \tbinom{n}{0}}C.$$
	Inverting we get: 
	$$[b,\underbrace{a,...,a}_n]^{-1}=b_n^{(-1)^1 \tbinom{n}{0}}...b_{0}^{(-1)^{n+1}\tbinom{n}{n}}C^{-1}.$$
	We then multiply these two expressions:
		$$[b,\underbrace{a,...,a}_n]^{-1}[b,\underbrace{a,...,a}_n]^a= 
		b_n^{(-1)^1 \tbinom{n}{0}}... b_{0}^{(-1)^{n+1}\tbinom{n}{n}}C^{-1}
		\cdot 
		b_{1}^{(-1)^n\tbinom{n}{n}}...b_{n+1}^{(-1)^0 \tbinom{n}{0}}C.$$
	 We can rearrange the terms of this product. This is done at the cost of adding commutators $[b_i,b_j]$ with   $\vert i-j\vert <n+1$. This  changes the central  element $C$ to a $C'$. Then, using the rules of Pascal's triangle: 
		$$\tbinom{n}{k}+\tbinom{n}{k+1}=\tbinom{n+1}{k+1},$$
		we see that the formula is indeed preserved  at each step. 
	
	We then consider $f_{n+1}=[[b,\underbrace{a,...,a}_{n+1}],b]$. Using the formula for $[b,\underbrace{a,...,a}_{n+1}]$ together with the formulas   $[x z, y] = [x, y]^z \cdot [z, y]$ and $\left[x^{-1}, y\right] = [y, x]^{x^{-1}}$, and the fact that we already know by induction  that each $d_i$ is central, for $i\le n$ , we see that 
	$$f_{n+1}=d_1^{(-1)^{n+1}\tbinom{n+1}{1}}...  d_{n}^{(-1)^{2}\tbinom{n+1}{n}}d_{n+1}^{(-1)^{1}\tbinom{n+1}{n+1}}.$$

	This shows again that $f_{n+1}$ and $d_{n+1}$ differ by a product of central elements. Thus one is central if and only if the other one is, and thus the relation $ [b,\underbrace{a,...,a}_{n+1},b,a]=1$ and $ [b,\underbrace{a,...,a}_{n+1},b,b]=1$ of $\pi_2$ are equivalent, modulo the previous relations, to the relations $[b_i,b_{i+n+1},a]$ and $ [b_i,b_{i+n+1},b]$ of $\pi_1$. 
\end{proof}

\begin{corollary}\label{lem: presentation Nilp Quotient}
	The group $G/\gamma_n(G)$ (marked by the images of $a$ and $b$ in the quotient) is given by the following presentations: \\ 
	For $n=2$: 	$$\langle a,b \, \vert \, [b,a]=1\rangle$$ 
	For $n=3$: $$\langle a,b \, \vert \, [b,a,a]=1,[b,a,b]=1\rangle$$ 
	For $n\ge 4$: 
		$$\langle a,b \, \vert \, \forall k \le n-3,\, [b,\underbrace{a,...,a}_k,b,a]=1,\, [b,\underbrace{a,...,a}_k,b,b]=1,  [b,\underbrace{a,...,a}_{n-1}]=1,  [b,\underbrace{a,...,a}_{n-2},b]=1\rangle$$ 

\end{corollary}

\begin{proof}
	The presentation for the quotient $G/\gamma_n(G)$ is obtained by removing, in the  presentation given by Lemma \ref{lem: presentation Hall Gp}, every iterated commutator of length more than $k$, and then adding  two iterated commutators: 
	$$[b,\underbrace{a,...,a}_{n-1}]=1 \text{ and }  [b,\underbrace{a,...,a}_{n-2},b]=1$$
	To check that this is indeed correct, we only need to check that adding these two relations indeed yields a class $n$ nilpotent group. But this is immediate: any iterated commutator of length $n$ that contains more than three $b$'s is trivial by the original relations of $G$, it thus  suffices to deal with those that contain exactly one or two $b$'s to have a $n$ nilpotent group. Dealing with these is precisely what the last two relations do.  
\end{proof}

\begin{lemma}
	The elements $f_n=[b,\underbrace{a,...,a}_n,b]$ form another basis of the centre of $G$. 
\end{lemma}
\begin{proof}
This follows directly from the computations that appear in the proof of Lemma \ref{lem: presentation Hall Gp}: for each $n>0$, the element  $f_n$ can be written as a product $d_1^{\alpha_1}...d_n^{\alpha_n}$ with $\alpha_n=-1$. Since the elements $d_n$ form a basis of the centre of $G$, this shows that so do the elements $f_n$.  
\end{proof}	
	
	\begin{lemma}\label{lem: image of odd form free basis}
		The image of the elements $f_1$, $f_3$, ..., $f_{2n-3}$  of $G$ in $G/\gamma_{2n}(G)$ form the basis of a free abelian group.		
	\end{lemma}
	
	\begin{lemma}\label{lem: residuall nilpotency of G}
	The group $G$ is residually nilpotent. 
\end{lemma}

These two lemmas are proved by using an explicit matrix group  which is a quotient of the group 
$G/\gamma_n(G)$, and which allows us to prove that some elements of $G/\gamma_n(G)$ are different from the identity. 

	\begin{lemma}\label{lem: Matrix model for some relations in G}
	Fix $n$ an integer greater than $3$. 
	Consider the upper triangular  $n\times n$ matrices $a$, $b$ and $c$:
	$$ a=\left(\begin{array}{cccccc}
		1 & 1 & 0 & \cdots & \cdots & 0\\
		& 1 & 1 &  &  & \vdots\\
		&  & \ddots & \ddots &  & \vdots\\
		&  &  & \ddots & 1 & 0\\
		&  &  &  & 1 & 1\\
		&  &  &  &  & 1
	\end{array}\right)$$
	$$ b=\left(\begin{array}{cccccc}
		1 & 1 & 0 & \cdots & \cdots & 0\\
		& 1 & 0 &  &  & \vdots\\
		&  & \ddots & \ddots &  & \vdots\\
		&  &  & \ddots & 0 & 0\\
		&  &  &  & 1 & 1\\
		&  &  &  &  & 1
	\end{array}\right)
	 $$ 
$$	c=\left(\begin{array}{ccccc}
	1 & 0 & \cdots & 0 & 1\\
	& 1 & \ddots &  & 0\\
	&  & \ddots & \ddots & \vdots\\
	&  &  & 1 & 0\\
	&  &  &  & 1
\end{array}\right)$$

	Then, using the same notation $f_k=[b,\underbrace{a,...,a}_k,b]$ as in $G$, we have that:
	$$f_k= c^{(-1)^n\binom{n-4}{k-1}} \text{ if }k<n-3,$$ $$ f_{n-3}=I_n\text{ if $n$ is odd,}$$ 
	 $$ f_{n-3}=c^2\text{ if $n$ is even.}$$
	 And $[b,\underbrace{a,...,a}_{n-2}]$ is central (it is also either $c^2$ if $n$ is even or $I_n$ if $n$ is odd.) 
	 Thus the group $M_n$ generated by $a$ and $b$ is a marked quotient of $(G/\gamma_{n}(G),(a,b))$. 
\end{lemma}
	
	We first show how this lemma allows to prove Lemma \ref{lem: image of odd form free basis} and Lemma \ref{lem: residuall nilpotency of G}. 
	\begin{proof}[Proof of Lemma \ref{lem: image of odd form free basis}]
		By Lemma  \ref{lem: Matrix model for some relations in G}, $f_{2n-3}$ has infinite order in $G/\gamma_{2n}(G)$: indeed its image in an explicitly given matrix group has infinite order. 
		Suppose that a relation holds between the elements $f_1$, $f_3$, ..., $f_{2n-3}$  of  $G/\gamma_{2n}(G)$. Written additively, we have an equation:
		  $$\alpha_1 f_1 + ...+\alpha_{2n-3}f_{2n-3}=0.$$
		  But the image of this relation in  $G/\gamma_{4}(G)$ is just 
	 $$\alpha_1 f_1 =0,$$ and since we know that $f_1$ has infinite order there, $\alpha_1=0$. Taking consecutive quotients in  $G/\gamma_{2i}(G)$, we see that the information that each $f_{2i-3}$ has infinite order in $G/\gamma_{2i}(G)$ is sufficient to show that the elements $f_1$, $f_3$, ..., $f_{2n-1}$  form the basis of a free abelian group. 
	\end{proof}
	
	\begin{proof}[Proof of Lemma \ref{lem: residuall nilpotency of G}]
	We have that $G/ \mathcal{Z}(G)=\Z \wr \Z$, and $\Z \wr \Z$ is residually nilpotent (for instance any non-identity element of  $\Z \wr \Z$ has a non-identity image in a quotient of the form $\Z/p \wr \Z/p$, for $p$ a big enough prime. This group is a $p$-group, thus it is nilpotent).\\
Thus all we have to show is that every non-identity element of $\mathcal{Z}(G)$ has a non-trivial image in a nilpotent quotient. 
Suppose that a relation   $$\alpha_1 f_1 + ...+\alpha_{k}f_{k}=0$$ holds in every quotient $G/\gamma_{i}(G)$, with $\alpha_k\ne 0$. 
This relation should thus hold in the groups $M_n$ given by Lemma  \ref{lem: Matrix model for some relations in G}. But by Lemma  \ref{lem: Matrix model for some relations in G} the image of  $f_{k}$  in the group $M_n$ is the matrix $f_k=\pm c^{\binom{n-4}{k-1}}$, whose coefficients grow as $n^{k-1}$ when $n$ goes to infinity, whereas the image of $\alpha_1 f_1 + ...+\alpha_{k-1}f_{k-1}$ is a matrix whose coefficients form a polynomial of degree $k-2$ in $n$. Thus no such relation can hold in every group $M_n$.
	\end{proof}
	
	\begin{proof}[Proof of Lemma \ref{lem: Matrix model for some relations in G}]
		Fix $n>3$. Denote by $I_n$ the $n\times n$ identity matrix and by $E_{i,j}$ the $n\times n$  matrix with only one coefficient $1$ in position $(i ,j)$. \\
		Thus $a=I_n+\sum_{i=1..n-1} E_{i,i+1}$, $b=I_n+ E_{1,2}+ E_{n-1,n}$ and $c=I_n+E_{1,n}$. The element $c$ is in the centre of the group of upper triangular matrices with ones on the diagonal. 
		Denote $u =\sum_{i=1..n-1} E_{i,i+1}$  so that $a=I_n+u$ and  $a^{-1}=I_n+\sum_{i=1..n-1} (-1)^i u^i$.

		We first compute the coefficients of the element $g_k=[b,\underbrace{a,...,a}_{k}]$, except the top right coefficient in position $(1,n)$, since it corresponds to a central element that cancels out when computing $f_n=[g_n,b]$. In what follows, we denote $A=_{/c} B$ to mean that two matrices agree except maybe on the coefficient of $(1,n)$. \\
				Note first that $g_1=[b,a]=_{/c} I_n+E_{1,3}+\sum_{i=2..n-2} (-1)^{n-i+1} E_{i,n}$. 
				\\
				Indeed: 
	\begin{align*}
	[b,a] & =_{/c}(I_n-E_{1,2}-E_{n-1,n})(I_n+\sum_{i=1..n-1} (-1)^i u^i)(I_n+E_{1,2}+E_{n-1,n})(I_n+u) \\ 
	& =_{/c}(I_n-E_{1,2}-E_{n-1,n})(I_n+\sum_{i=1..n-1} (-1)^i u^i)(I_n+E_{1,2}+E_{1,3}+E_{n-1,n}+u) \\
	& =_{/c}(I_n-E_{1,2}-E_{n-1,n})(I_n+E_{1,2}+E_{1,3}+E_{n-1,n}+(\sum_{i=1..n-1} (-1)^i u^i)E_{n-1,n}) \\
	& =_{/c} I_n+E_{1,3}+(\sum_{i=1..n-1} (-1)^i u^i)E_{n-1,n}-E_{1,2}(\sum_{i=1..n-1} (-1)^i u^i)E_{n-1,n} \\
	& =_{/c} I_n+E_{1,3}+(\sum_{i=1..n-1} (-1)^i u^i) E_{n-1,n} \text{ (the last term contributes to the centre)}
\end{align*}
And $(\sum_{i=1..n-1} (-1)^i u^i) E_{n-1,n}= \sum_{i=1..n-2} (-1)^i E_{n-1-i,n}= \sum_{i=1..n-2} (-1)^{n-i+1} E_{i,n}$. \\

We then prove by induction that for $k\le n-3$
				$$g_k=_{/c} I_n+E_{1,k+2}+\underset{i=2..n-2}{\sum} (-1)^{n-i+1} \binom{n-i-2}{k-1}E_{i,n}.$$
				(Here  $\binom{p}{t} =0$ when $t>p$.)
	Denote  $h_k=\underset{i=2..n-2}{\sum} (-1)^{n-i+1} \binom{n-i-2}{k-1}E_{i,n}$ so that $g_k=I_n+E_{1,k+2}+h_k$.
	\\
	Now $g_k^{-1}=_{/c}I_n-E_{1,k+2}-h_k$. 	We now suppose that $g_k$ has the desired form, and compute $g_{k+1}=[g_k,a]$:
	\begin{align*}
		[g_k,a] & =_{/c}(I_n-E_{1,k+2}-h_k)(I_n+\sum_{i=1..n-1} (-1)^i u^i)(I_n+E_{1,k+2}+h_k)(I_n+u)\\ 
	& =_{/c}(I_n-E_{1,k+2}-h_k)(I_n+\sum_{i=1..n-1} (-1)^i u^i)(I_n+E_{1,k+2}+E_{1,k+3}+h_k+u) \\
	& =_{/c}(I_n-E_{1,k+2}-h_k)(I_n+E_{1,k+2}+E_{1,k+3}+h_k+(\sum_{i=1..n-1} (-1)^i u^i) h_k) \\
	& =_{/c} I_n+E_{1,k+3}+(\sum_{i=1..n-1} (-1)^i u^i) h_k-E_{1,k+2}h_k-(E_{1,k+2}+h_k)(\sum_{i=1..n-1} (-1)^i u^i) h_k \\
	& =_{/c} I_n+E_{1,k+3}+(\sum_{i=1..n-1} (-1)^i u^i) h_k \text{ (deleting central  terms)}
\end{align*}

Thus all that is left to be proved is  the equality  $$(\sum_{i=1..n-1} (-1)^i u^i) (\underset{i=2..n-2}{\sum} (-1)^{n-i+1} \binom{n-i-2}{k-1}E_{i,n}) =_{/c}\underset{i=2..n-2}{\sum} (-1)^{n-i+1}  \binom{n-i-2}{k}E_{i,n}.$$
The coefficient of $E_{i,n}$ obtained by expanding the product on the left is given by products of the form $((-1)^p u^p)((-1)^{n-i-p+1} \binom{n-i-p-2}{k-1}E_{i+p,n})$. Summing, we see that the coefficient of $E_{i,n}$, $2\le i\le n-1$, is given by:
$$ \underset{p=1..n-i-2}{\sum}(-1)^{n-i+1} \binom{n-i-p-2}{k-1}=\underset{t=1..n-i-3}{\sum}(-1)^{n-i+1} \binom{t}{k-1}$$
The results then follows from the following formula for binomial coefficients:
$$\binom{t}{p+1}=\underset{j=1..t-1}{\sum}\binom{j}{p}$$ 
\\
The formula for $g_k$ (modulo the centre)  is thus established. \\
We now compute $f_k=[g_k,b]$ (exactly)  to end the proof of the lemma. Note that $g_k^{-1}=I_n-E_{1,k+2}-h_k+E_{1,k+2}h_k$.
\\
Suppose first that $k<n-3$, so that $E_{1,k+2}E_{n-1,n}=0$. 
\begin{align*}
	[g_k,b] & =g_k^{-1}(I_n-E_{1,2}-E_{n-1,n})(I_n+E_{1,k+2}+h_k)(I_n+E_{1,2}+E_{n-1,n})\\ 
	& =g_k^{-1}(I_n-E_{1,2}-E_{n-1,n})(I_n+E_{1,k+2}+h_k+E_{1,2}+E_{n-1,n}) \\
	& =(I_n-E_{1,k+2}-h_k+E_{1,k+2}h_k)(I_n+E_{1,k+2}+h_k-E_{1,2} h_k) \\
	& = I_n-E_{1,2} h_k \\
	&=I_n+(-1)^{n}\binom{n-4}{k-1}E_{1,n}= c^{(-1)^n\binom{n-4}{k-1}} 
\end{align*}
In case $n-3=k$, there are two possibilities, depending on whether $n$ is odd or even. We have   $b=I_n+ E_{1,2}+ E_{n-1,n}$  and $g_{n-3}=I_n+ E_{1,n-1}+ (-1)^{n-1}E_{2,n}$.
An easy computation yields $$f_{n-3}=[g_{n-3},b]=(I_n+(1+(-1)^{n}) E_{1,n}),$$
and thus $f_{n-3}=c^2$ if $n$ is even or $I_n$ if $n$ is odd. 
This ends the lemma. 
\end{proof}
	
	\subsection{Construction of a group with non-computable nilpotent quotients }
	
We can now give a full proof of the main theorem of this section. 

\begin{proof}[Proof of Theorem \ref{Thm: uncomputable nilp quotients}]
	We use the notations introduced before: $G$ is Hall's group, and its center is a free abelian group generated by elements $f_n$, $n\in\N$.\\
	We define a set $\mathcal{A}$ of relations for $G$ defined in terms of Turing machines. \\Denote by $p_n$ the sequence of odd prime numbers: $p_1=3$, $p_2=5$, $p_3=7$...    
	Consider an effective enumeration of all Turing Machines: $M_1$, $M_2$, $M_3$... 
	For each $n$, consider a run of $M_n$ with no input. 
	If this run never stops, add nothing to $\mathcal{A}$. If this run stops in $t$ steps, we add a relation $f_{p_n}=f_{p_n^t}$ to $\mathcal{A}$.
	
	Denote by $H$  the group $G/\mathcal{A}$.  
	
\begin{lemma}
	The group $H$ has solvable word problem. 
\end{lemma}

\begin{proof}
	Starting from Hall's group $G$, the group $H$ is obtained by adding a recursive set $\mathcal{A}$ of relations. Those relations  identify some central elements. 
	Just as $G$, $H$ maps onto the wreath product $\Z \wr \Z$, by collapsing the centre. Since $\Z \wr \Z$ has solvable word problem, to solve the word problem in $H$, it suffices to be able to solve the word problem for elements in the kernel of the morphism $H\rightarrow \Z \wr \Z$, thus in the centre of $H$. \\
	The centre $\mathcal{Z}(G)$ of $G$ is $\bigoplus_{n\in\N}\Z$ with basis the elements $f_i$, $i\ge 1$. 
	In $H$, this centre is quotiented by relations  $f_{p_n}=f_{p_n^t}$, for each $n$ and $t$ such that the $n$-th Turing machine stops in $t$ steps on no input. 
	Thus the centre  $\mathcal{Z}(H)$ is obtained from that of $G$ by identifying disjoint copies of $\Z$: it is still isomorphic to $\bigoplus_{n\in\N}\Z$.
	\\
	To determine whether a word $w$ on the $f_i$ is the identity in $\mathcal{Z}(H)$, one in fact need only to take into account a finite portion of the relations of $\mathcal{A}$: those that will identify some $f_i$ with some $f_j$, when $i$ and $j$ both appear in $w$. 
	\\
	Denote by $N$ the greatest number such that $f_N$ is a letter of $w$. We thus have to take into account all relations  $f_{p_n}=f_{p_n^t}$, where $p_n\le N$ and $p_n^t\le N$.   
	The list of all such relations can be obtained by running all Turing machines $M_1$, ..., $M_N$, for at most $N$ steps. This takes a finite amount of time. 
	To determine whether $w=1$ in $\mathcal{Z}(H)$, one is then left with a finitely generated abelian group given by a finite presentation, where the word problem can easily be solved. 
\end{proof}

\begin{lemma}\label{lem: uncomputable nilp quotients}
	There is no algorithm that can, on input $n$, produce a finite presentation of $H/\gamma_n (H)$.
\end{lemma}

\begin{proof}
	This follows directly from the following fact. 
By  Lemma \ref{lem: image of odd form free basis}, the element $f_{p_n}$ is mapped to a non-identity element of $G/\gamma_{k}(G)$ if and only if $k\ge p_n+3$. 
In $H$, this element is identified with an element $f_i$, with $i>p_n$, if and only if the $n$-th Turing machine halts. 
The relations of $H$ that depend on Turing machines other that the $n$-th cannot have an influence on $f_{p_n}$: we also have that $f_{p_n}=1$ in $H/\gamma_{p_n +3}(H)$ if and only if the $n$th Turing machine halts. \\
This implies immediately that the word problem is not uniformly solvable in the groups $H/\gamma_{k}(H)$, as $k$ varies. Since finitely presented nilpotent groups have uniformly solvable word problem, this also implies that no algorithm can, given $k$, output a finite presentation for $H/\gamma_{k}(H)$. 
\end{proof}

\begin{lemma}\label{lem: H is res nilp}
	The group $H$ is residually nilpotent. 
\end{lemma}

\begin{proof}
	As before, we have that $H/ \mathcal{Z}(H)=\Z \wr \Z$, and $\Z \wr \Z$ is residually nilpotent.\\
	Thus all we have to show is that every non-identity element of $\mathcal{Z}(H)$ has a non-trivial image in a nilpotent quotient. 
	\\
	Take a word $w$ which is a product of elements $f_i$, $1\le i \le N$ and which defines a non-identity element of $H$.\\
	Each  element $f_i$ that appears in this word is possibly identified in $H$ with another central element $f_j$ via a relation $f_{p_n}=f_{p_n^t}$. 
	Consider the finite set of words $W$ which consists in all words obtained by substituting  in $w$ some elements $f_i$ by  elements $f_j$ to which they are equal. The set $W$ is finite since its size is at most $2^{|w|}$. 
	All these words define the same element as $w$ in $H$, but they define different elements in $G$. However, they all define non trivial elements in $G$. By Lemma \ref{lem: residuall nilpotency of G}, $G$ is residually nilpotent, thus there exists $M$ such that each element of $W$ is mapped to a non-identity  element of $G/\gamma_M (G)$. 
	\\
	It then follows that $w$ has a non-identity image in $H/\gamma_M (H)$, and thus that $H$ is indeed residually nilpotent. 
	
	(Note that the set $W$ can be enumerated, but there is no way to determine that we are done enumerating it. 
	 And thus  the number $M$  that appears in this proof cannot be computed.)
\end{proof}

This ends the proof of Theorem \ref{Thm: uncomputable nilp quotients}
\end{proof}

\begin{proof}[Proof of Theorem \ref{Thm: non effectively residually nilp} ]
	This follows directly from the proof of Lemma \ref{lem: uncomputable nilp quotients}: the information  ``$f_{p_k}$ has a non-trivial image in $H/\gamma_{n}(H)$'' exactly says that either the $k$-th Turing machine never halts, or that, if it does, it does so in $t$ steps, with $p_k^t\le n-3$. This information is sufficient to solve the halting problem, and thus it cannot be produced by an algorithm. 
\end{proof}

\newpage
\printbibliography

\end{justify}
\end{document}